\documentclass{amsart} 
\usepackage{amssymb,times, amscd,amsmath,amsthm, xypic}
\author{Shoji Yokura$^{*}$}
\address
{Department of Mathematics and Computer Science, 
Faculty of Science, 
Kagoshima University, 21-35 Korimoto 1-chome, Kagoshima 890-0065, Japan}
\email {yokura@sci.kagoshima-u.ac.jp}
\title
{ Motivic Characteristic Classes$^{\dagger}$}
\thanks {$\dagger$ An expanded version of the author's talk at the workshop ``Topology of Stratified Spaces" held at MSRI, Berkeley, from September 8 to 12, 2008 }
\thanks {* Partially supported by Grant-in-Aid for Scientific Research
(No. 21540088), the Ministry of Education, Culture, Sports, Science and Technology (MEXT), and JSPS Core-to-Core Program 18005, Japan}

\begin{document} 
\numberwithin{equation}{section}
\newtheorem{thm}[equation]{Theorem}
\newtheorem{pro}[equation]{Proposition}
\newtheorem{prob}[equation]{Problem}
\newtheorem{cor}[equation]{Corollary}
\newtheorem{con}[equation]{Conjecture}
\newtheorem{ob}[equation]{Observation}
\newtheorem{lem}[equation]{Lemma}
\theoremstyle{definition}
\newtheorem{ex}[equation]{Example}
\newtheorem{defn}[equation]{Definition}
\newtheorem{rem}[equation]{Remark}
\renewcommand{\rmdefault}{ptm}
\def\alp{\alpha}
\def\be{\beta}
\def\jeden{1\hskip-3.5pt1}
\def\om{\omega}
\def\bigstar{\mathbf{\star}}
\def\ep{\epsilon}
\def\vep{\varepsilon}
\def\Om{\Omega}
\def\la{\lambda}
\def\La{\Lambda}
\def\si{\sigma}
\def\Si{\Sigma}
\def\Cal{\mathcal}
\def\ga{\gamma}
\def\Ga{\Gamma}
\def\de{\delta}
\def\De{\Delta}
\def\bF{\mathbb{F}}
\def\bH{\mathbb H}
\def\bPH{\mathbb {PH}}
\def \bB{\mathbb B}
\def \bA{\mathbb A}
\def \bOB{\mathbb {OB}}
\def \bM{\mathbb M}
\def \bOM{\mathbb {OM}}
\def \calB{\mathcal B}
\def \bK{\mathbb K}
\def \bG{\mathbf G}
\def \bL{\mathbf L}
\def\bN{\mathbb N}
\def\bR{\mathbb R}
\def\bP{\mathbb P}
\def\bZ{\mathbb Z}
\def\bC{\mathbb  C}
\def \bQ{\mathbb Q}
\def\op{\operatorname}

\maketitle

\begin{abstract} Motivic characteristic classes of possibly singular algebraic varieties are homology class versions of motivic characteristics, not classes in the so-called motivic (co) homology. This paper is a survey on them with more emphasis on capturing infinitude finitely and on the motivic nature, in other words, the scissor relation or additivity.
\end{abstract}

\section {Introduction} 

Characteristic classes are usually cohomological objects defined on real or complex vector bundles, thus for any smooth manifold, characteristic classes of it are defined through its tangent bundle. For the real vector bundles, Stiefel--Whitney classes and Pontraygin classes are fundamental ones and in the complex vector bundles the Chern class is the fundamental one. When it comes to a non-manifold space, such as a singular real or complex algebraic or analytic variety, one cannot talk about its cohomological characteristic class, unlike the smooth case, simply because one cannot define its tangent bundle, although one can define some reasonable substitutes, such as tangent cone, tangent star cone, which are not vector bundles, rather stratified vector bundles. In the 1960's people started to define characteristic classes on algebraic varieties as homological objects, not through some kind of vector bundles considered on them, but as higher analogues of some geometrically important invariants such as Euler--Poincar\'e characteristic, signature, etc. I suppose that the theory of characteristic classes of singular spaces start with Thom's $L$-class for oriented $PL$-manifolds \cite{Thom}, whereas Sullivan's Stiefel--Whitney classes and the so-called Deligne--Grothendieck conjecture about the existence of Chern homology classes started the whole story of \emph{capturing characteristic classes of singular spaces as natural transformations}, more precisely as a natural transformation from a certain covariant functor to the homology functor.  The Deligne--Grothendieck conjecture seems to be based on Grothendieck's idea or modifying his conjecture on a \emph{Riemann--Roch type formula} concerning the constructible \'etale sheaves and Chow rings (cf. \cite[Part II, note(87$_1$), p.361 ff.]{Grot}) and made in the well-known form by P. Deligne later. R. MacPherson \cite{M1} gave a positive answer to the Deligne--Grothendieck conjecture and,  motivated by this solution,  P. Baum, W. Fulton and R. MacPherson \cite{BFM} further established the singular Riemann--Roch Theorem, which is a singular version of Grothendieck--Riemann--Roch, which is a functorial extension of the celebrated Hirzebruch--Riemann--Roch (abbr.  HRR) \cite{Hi}. HRR is the very origin of Atiyah--Singer Index Theorem.\\

The main resuts of \cite{BSY1} (announced in \cite{BSY2}) are the following:
\begin{itemize}
\item {\bf ``Motivic" characteristic classes of algebraic varieties}, which is a class version of the motivic characteristic. (Note that this ``motivic class" is \underline {not} a class in the so-called ``motivic cohomology" in algebraic/arithmetic geometry.)\\

\item Motivic characteristic classes in a sense give rise to {\bf a ``unification" of three well-known important characteristic homology classes}:

\begin{enumerate}

\item MacPherson's Chern class transformation \cite{M1} (cf. \cite{M2}, \cite{Sw}, \cite{BrS}), 

\item Baum--Fulton--MacPherson's Riemann--Roch transformation \cite{BFM}

\item Goresky--MacPherson's L-homology class \cite {GM} or Cappell--Shaneson's L-homology class \cite{CS1}(cf. \cite{CS2})\\
\end{enumerate}
\end{itemize}

This unification result can be understood  to be good enough to consider our motivic characteristic classes  as a positive solution to the following MacPherson's question or comment (written at the end of his survey paper of 1973 \cite{M2}):\\

\emph{``It remains to be seen whether there is a unified theory of characteristic classes of singular varieties like the classical one outlined above."}\\

It unifies ``only three" characteristic classes, though, but so far it seems to be a reasonble one.

The purpose of the present paper is mainly to explain the above results of \cite{BSY1}(also see \cite{SY}) with putting more emphasis on ``motivic nature" of motivic characteristic classes. In particular, we show that our motivic characteristic class is a very natural class version of the so-called motivic characteristic, just like the way A. Grothendieck extened HRR to Grothendieck --Riemann--Roch. For that, we go back all the way to the natural numbers, which would be thought of  as the very ``origin" of \emph {characteristic} or \emph {characteristic class}. We na\"\i vely start with the simple counting of finite sets. Then we want to count infinite sets as if we are still doing the same way of counting finite sets,  and want to understand motivic characteristic classes as higher class versions of this unusual ``counting infinite sets", where infinite sets are complex algebraic varieties. (The usual counting of infinite sets, forgetting the structure of a variety at all, lead us into ``mathematics of infinity".)  The key is Deligne's mixed Hodge structures \cite{De1, De2} or more generally Saito's deep theory of mixed Hodge modules \cite{Sai2}, etc. \\

As to the Mixed Hodge Hodules (abbr. MHM), in \cite{Sch3} J\"org Sch\"urmann gives a very nice introduction and overview about recent developments on the interaction of theories of characteristic classes and Mixed Hodge Theory for singular spaces in the complex algebraic context with MHM as a crucial and fundamental key. For example, a study  of characteristic classes of the intersection homological Hodge modules have been done in a series of papers by Sylvain Cappell, Anatoly Libgober, Laurentiu Maxim, J\"org Sch\"urmann and Julius Shaneson \cite {CLMS1, CLMS2, CMS1, CMS2, CMSS, MS, MS2} (as to \cite{MS2} also see \cite{Yokura-milnor}). \\

The very recent book by C. Peters and J. Steenbrink \cite{PS} seems to be a most up-to-date survey on mixed Hodge structures and Saito's mixed Hodge modules. The Tata Lecture Notes by C. Peters \cite{P} (which is a condensed version of \cite{PS}) gives a nice introduction to Hodge Theory with more emphasis on the motivic nature \footnote{J. Sch\"urmann informed me of the book \cite{PS} and the lecture \cite{P} at the workshop. }.

\section{Preliminaries: from natural numbers to genera}

First of all let us consider counting the number of elements of finite sets, i.e., natural numbers. Let $\Cal {FSET}$ be the category of finite sets and maps among them. For an object $X \in \Cal {FSET}$, let
$$ c(X) \in \bZ $$
be the number of the elements of $X$,  which is usually denoted by $|X|$ ($\in \bN$) and called the cardinal number, or cardinality of $X$. It satisfies the following four properties on the category $\Cal {FSET}$ of finite sets: 
\begin{enumerate}

\item
$\displaystyle  X \cong X'$ (bijection or equipotent) $\Longrightarrow$ $c(X) = c(X')$,

\item
$c(X) = c(X \setminus Y) + c(Y)$ for  $Y \subset X$, 

\item
$c(X \times Y) = c(X)\cdot c(Y)$,

\item
$c(pt) = 1$. \quad (Here $pt$ denotes one point.) 
\end{enumerate}

\begin{rem} \label{c-rem}Clearly these four properties characterize the counting $c(X)$. Also note that if $c(X) \in \bZ$ satisfies  (1) -- (3) without (4), then we have $c(pt) = 0$ or $c(pt) = 1$.
If $c(pt) = 0$, then it follows from (2) (or (1) and (3)) that $c(X) = 0$ for any finite set $X$.
If $c(pt) =1$, then it follows from (2) that  $c(X) = $ the number of elements of a finite set $X$.
\end{rem}

\begin{rem}  When it comes to infinite sets, then the cardinality still satisfies the above four properties, however the usual ``computation" does not work any longer; such as $a^2 =a \Longrightarrow a = 0 \, \, \, \text {or} \,  \, \, 1$. For example, for any natural number $n$,
$$c(\bR^n) = c(\bR) \, \text {, i.e., denoted by, } \, \aleph^n = \aleph.$$
Namely, we enter \emph{the mathematics of infinity}.  Generalizing the above, we could still consider the above ``counting" on the bigger category $\Cal {SET}$ of sets, i.e., a set can be infinite, and $c(X)$ in a certain integral domain. However, one can see that there does not exist such counting; in fact one can see that if such a counting exists so that (1), (2) and (3) are satisfied, then it automatically follows that $c(pt) = 0$, which contradicts to the property (4). Thus the upshot is:\

\emph{If we consider the above counting on the category $\Cal {SET}$ of not-necessarily-finite sets, then such a counting automatically has to be a trivial one, i.e., $c(X) = 0$ for any set $X$ !}

\emph{However, if we consider sets having superstructures on the infrastructure (= set) and the property (1) is replaced by the invariance of the superstructures, then we will obtain more reasonable countings which are finite numbers, thus we can avoid the mysterious ``mathematics of infinity" and  extend the usual counting $c(X)$ of finite sets very naturally and na\" \i vely. This is nothing but what is all about the Euler characteristic, genera, etc., which are basic, important and fundamental objects in modern geometry and topology.}
\end{rem}

 Let us consider the following ``topological counting" $c_{top}$ on the category $\Cal {TOP}$ of topological spaces, which assigns to each topological space $X$ a certain integer (or more generally, an element in an integral domain) 
$$c_{top}(X) \in \bZ $$ 
such that it satisfies the following four properties, which are exactly the same as above except for (1):
\begin{enumerate}
\item
$ X\cong X'$ (homeomorphism = $\Cal{TOP}$- isomorphism) $\Longrightarrow$ $c_{top}(X) = c_{top}(X')$,

\item
$c_{top}(X) = c_{top}(X\setminus Y) + c_{top}(Y)$ for $Y \subset X$ (for the moment no condition), 

\item
$c_{top}(X \times Y) = c_{top}(X)\cdot \chi_{top}(Y)$,

\item
$c_{top}(pt) = 1.$ 
\end{enumerate} 

\begin{rem}
As in the above Remark(\ref{c-rem}) ,  (1) and (3) imply that $c_{top}(pt) = 0$ or $1$. If $c(pt) = 0$, then it follows from (1) and (3) that $c_{top}(X) = 0$ for any topological space $X$. Thus the last condition (4) $c(pt) = 1$ means that $c_{top}(X)$ is \underline {a nontrivial one}. Hence, the topological counting $c_{top}$ can be put in as \emph{a nontrivial, multiplicative, additive, topological invariant}.
\end{rem} 

%

\begin{pro} \label{c_top} If such a $c_{top}$ exists, then we must have that
$$c_{top}(\bR^1) =  -1, \quad \text {hence} \quad c_{top}({\bR}^n) = (-1)^n.$$
Hence if $X$ is a finite $CW$-complex with $\sigma_n(X)$ denoting the number of open $n$-cells, then
$$c_{top}(X) = \sum_n (-1)^n \sigma_n(X) = \chi(X)$$
is the Euler--Poincar\'e characteristic of $X$. 
\end{pro}
The equality $c_{top}(\bR^1) =  -1$ can be seen as follows: Consider 
$$\bR^1 = (-\infty, 0) \sqcup \{0 \} \sqcup (0, \infty).$$
Which implies that
$$c_{top}(\bR^1) = c_{top}((-\infty, 0)) + c_{top}(\{0 \}) + c_{top}((0, \infty)).$$
Hence we have
$$-c_{top}(\{0 \}) = c_{top}((-\infty, 0)) + c_{top}((0, \infty)) - c_{top}(\bR^1).$$
Since $\bR^1 \cong (-\infty, 0) \cong (0, \infty)$, it follows from (1) and (4) that
$$ c_{top}(\bR^1) = -c_{top}(\{0 \}) = -1.$$

To show the existence of such a counting $c_{top}$, we use or need  Ordinary Homology/ Cohomology Theory:(symbolically speaking or as a slogan) \\

\hspace {2cm}  {\bf topological counting $c_{top}$ : Ordinary (Co)homology Theory} \\

To be more precise, we use the Borel--Moore homology theory \cite{Borel-Moore}, which is defined to be the homology theory with closed supports. For a locally compact Hausdorff space $X$, the Borel--Moore homology theory $H_*^{BM}(X; R)$ with a ring coefficient $R$ is isomorphic to the relative homology theory of the pair $(X^c, *)$ with $X^c$ the one-point compactification of $X$ and $*$ the one point added to $X$:
$$H_*^{BM}(X; R) \cong H_*(X^c, *  ; R).$$
Hence, if $X$ is compact, the Borel--Moore homology theory is the usual homology theory: $H_*^{BM}(X; R) = H_*(X; R)$. 

Let $\frak K$ be a field (e.g., $\bR$ or $\bC$). If the Borel--Moore homology theory $H_*^{BM}(X; \frak K)$ is finite dimensional, e.g., if $X$ is finite $CW$-complex, then the Euler--Poincar\'e characteristic $\chi_{BM}$ using the Borel--Moore homology theory with a field coefficient $\frak K$ (e.g., $\bR$ or $\bC$)
$$\chi_{BM}(X) := \sum_n (-1)^n \op {dim}_{\frak K} H_{n}^{BM}(X; \frak K)$$
gives rise to the above topological counting $\chi_{top}$, because it satisfies that $H_n^{BM}(\bR^n, \frak K) = \frak K$ and $H_k^{BM}(\bR^n, \frak K) = 0$ for $k \not = n$, and thus
$$\chi_{BM}(\bR^n) = (-1)^n.$$
It turns out that for coefficients in the field $\frak K$, the Borel-Moore homology is \emph {dual \footnote { For an $n$-dimensional manifold $M$  the Poincar\'e duality map $\Cal {PD}: H_c^k(M) \cong H_{n-k}(M)$ is an isomorphism and also $\Cal {PD}: H^k(M) \cong H_{n-k}^{BM} (M)$ is an isomorphism. Thus they are \emph {Poincar\'e dual}, but \underline {not} \emph{dual as vector spaces}.} as a vector space} to the cohomology with compact support, namely 

$$H^{BM}_p(X; \frak K) = Hom(H_c^p(X; \frak K), \frak K).$$
Since $\frak K$ is a field, we have
$$H^{BM}_p(X; \frak K) \cong H_c^p(X; \frak K)$$
Hence the Euler-Poincar\'e characteristic using the Borel-Moore homology
$\chi_{BM}(X)$
 is equal to the Euler-Poincar\'e characteristic using the cohomology with compact support, usually denoted by $\chi_c$,
$$\chi_c(X) =\sum_i (-1)^i dim_K H^i_c(X; \frak K).$$
Since it is quite common to use $\chi_c$, we have

\begin{cor} For the category of locally compact Hausdorff spaces, 
$$c_{top} = \chi_c$$
the Euler--Poincar\'e characteristic using the cohomology with compact support.
\end{cor}

\begin{rem} The above story could be simply said as follows: There could exist infinitely many ways of ``topological counting" on the category $\Cal {TOP}$ of topological spaces, but they are \emph { \underline {all} identical to the Euler--Poincar\'e characteristic with compact support} when restricted to the subcategory of locally compact Hausdorff spaces with finite dimensional Borel-Moore homologies. Symbolically speaking, we can simply say that 
$$`` c_{top} = \chi_c ".$$
\end{rem}

Next let us consider the following ``algebraic counting" $c_{alg}$ on the category $\Cal {VAR}$ of \underline {complex} algebraic varieties (of finite type over $\bC$), which assigns to each complex algebraic variety $X$ a certain element 
$$c_{alg}(X) \in R $$ 
 in a commutative ring $R$ with unit $1$,  such that \
 
\begin{enumerate}
\item
$ X\cong X'$ ($\Cal {VAR}$-isomorphism) $\Longrightarrow$ $c_{alg}(X) = c_{alg}(X')$,

\item
$c_{alg}(X) = c_{alg}(X\setminus Y) + c_{alg}(Y)$ for a closed subvariety $Y \subset X$

\item
$c_{alg}(X \times Y) = c_{alg}(X)\cdot \chi_{alg}(Y)$,

\item
$c_{alg}(pt) = 1.$ \\
\end{enumerate} 

Just like $c(X)$ and $c_{top}(X)$, the last condition simply means that $c_{alg}$ is a nontrivial one.\\

The real numbers $\bR$ and in general the Euclidean space $\bR^n$ are the most fundamental objects in the category $\Cal {TOP}$ of topological spaces, the complex numbers $\bC$ and in general complex affine spaces $\bC^n$ are the most fundamental objects in the category $\Cal {VAR}$ of complex algebraic varieties.  The decomposition of the $n$-dimensional complex projective space
$$\bP^n = \bC^0\sqcup \bC^1 \sqcup \cdots \sqcup \bC^{n-1} \sqcup \bC^n$$
implies the following 
\begin{pro}\label{alg-c} If such a $c_{alg}$ exists, then we must have that
$$c_{alg}(\bP^n) = 1 - y + y^2 - y^3 + \cdots + (-y)^n$$
where $y := - c_{alg} (\bC^1) \in R$.
\end{pro}


\begin{rem} Proposition \ref{alg-c} already indicates that there could exist as infinitely many ways as the integers $y$'s  of ``algebraic counting" $c_{alg}$ on the category $\Cal {VAR}$ of complex algebraic varieties. Which is strikingly different from the ``topological counting" $c_{top}$ and the original counting $c$; in these cases they are uniquely determined. This difference of course lies in the ``complex structure":
$$\text { a set + a topological structure + {\bf a complex structure}.}$$
Certainly one cannot consider $\bR^1$ and thus the previous argument for $c_{top}(\bR^1) = -1$ DOES NOT work. In this sense, we should have used the symbol $c_{alg/\bC }$ to emphasize the complex structure, instead of $c_{alg}$. Since we are dealing with only the category of complex algebraic varieties in this paper, we just denote $c_{alg}$. See Remark \ref{real} below for the category of real algebraic varieties.
\end{rem}

To show the existence of such a $c_{alg}$, in fact, to show much more ways of counting than \underline {as infinitely many ways as the integers $y$'s}, we need or use the \emph {Deligne's Theory of Mixed Hodge Structures} \cite{De1, De2}, which comes from the algebraic structure.
$$\text { a set + a topological structure + {\bf a complex structure + an algebraic structure}.}$$

Then the Hodge--Deligne polynomial \\

$$\chi _{u,v}(X) := \sum_{i, p, q \geq 0} (-1)^i (-1)^{p+q}\op{dim}_{\bC} (Gr^p_F Gr^W_{p+q} H^i_c(X, \bC)) u^p v^q$$
satisfies the above four properties with $R = \bZ[u,v]$ and $-y := c_{alg}(\bC^1) = uv$, namely any Hodge--Deligne polynomial $\chi _{u,v}$ with $uv = -y$ is such a $c_{alg}$.   Here we point out that by Deligne's work only the graded terms with $p, q \geq 0$ are non-trivial, otherwise one would get $\chi_{u,v}(X) \in \bZ[u, u^{-1}, v, v^{-1}]$.

Similarly one can consider the invariant $c_{alg}(X) := \chi_{y, -1} \in \bZ[y]$ with $c_{alg}(\bC^1) = -y$. \\

Here we should note that when $(u, v) = (-1, -1)$, then we have
$$\chi_{-1, -1}(X) = \chi_c(X) = c_{top}(X).$$
Furthermore we note that for a smooth compact variety $X$ we have that 
\begin{itemize}

\item
$\chi_{0, -1}(X)$ is the arithmetic genus, 

\item
$\chi_{1, -1}(X)$ is the signature. 

\end{itemize}
These three cases $(u, v) = (-1, -1)$, $(0, -1)$ and $(1,-1)$ are very important ones.

$$\text {\bf algebraic counting $c_{alg}$ :   Mixed Hodge Theory \hspace {3cm}}$$
$$\hspace {3cm}\text {\bf  = Ordinary (Co)homology Theory}$$
$$\hspace {4cm} \text {\bf  +  \quad Mixed Hodge Structures}.$$\

\begin{rem} (e.g., see \cite{DK}) The following description is also fine, but we do the above one for the later discussion on motivic characteristic classes:
$$c_{alg}(\bP^n) = 1 +y + y^2 + y^3 + \cdots + y^n$$
where $y = c_{alg} (\bC^1) \in \bZ[y]$. The Hodge--Deligne polynomial is usually denoted by $E(X; u, v)$ and defined to be

$$E(X; u, v) := \sum_{i, p, q \geq 0} (-1)^i \op {dim}_{\bC} (Gr^p_F Gr^W_{p+q} H^i_c(X, \bC)) u^p v^q.$$
Thus we have
$$\chi _{u,v}(X)  = E(X; -u, -v).$$
The reason why we make such a modification lies in the definition of the Hirzebruch's generalized Todd class and Hirzebruch's $\chi_y$ characteristic, which will come below.
\end{rem}

\begin{con} \label{con}The ``algebraic counting" $c_{alg}$ specializes to the topological counting  $c_{top}$. Are there more ``algebraic counting" $c_{alg}$ which specialize to the Hodge--Deligne polynomial $\chi_{u,v}$ (which is sensitive to an algebraic structure) ? The answer would be negative. In other words, there would be no extra structures other than Deligne's mixed Hodge structure that contribute more to the algebraic counting $c_{alg}$ of complex algebraic varieties.\\
\end{con}

\begin{rem}\label{real}
 In the category $\Cal{VAR}(\bR)$ of \emph{\underline {real} algebraic varieties}, we can consider $c_{alg/\bR }(\bR^1)$ of the real line $\bR^1$, therefore we might be tempted to make a hasty conclusion that in the category of real algebraic varieties the topological counting $c_{top}$, i.e., $\chi_c$,  is sufficient. Unfortunately, the argument for $c_{top}(\bR^1) = -1$ DOES NOT work in the category $\Cal{VAR}(\bR)$, simply because $\bR^1$ and $(-\infty, 0)$ or $(0, \infty)$ are not isomorphic as real algebraic varieties. Even as compact varieties there DO exist real algebraic varieties which are homeomorphic but not isomorphic as real algebraic varieties; the following are such examples (see \cite[Example 2.7]{MP}):
 
The usual \emph{normal crossing} ``figure eight" curve:
$$F8 = \{(x, y) | y^2 = x^2 - x^4 \}.$$
The proper transform of $F8$ under the blowup of the plane at the origin is homeomorphic to a circle, and the preimage of the singular point of $F8$ is two points.

The \emph {tangential} ``figure eight" curve: 
$$tF8= \left \{(x, y) | \{(x+1)^2 + y^2 -1\}\{(x-1)^2+y^2-1\} = 0 \right \},$$
which is the union of two circles tangent at the origin. Therefore, in contrast to the category of crude topological spaces, in the category of \underline {real algebraic} varieties
an ``algebraic counting" $c_{alg}$ is meaningful, i.e., sensitive to an algebraic structure. Indeed, as such a ``real algebraic counting" $c_{\bR alg}$ there are
$$\emph {the $i$-th virtual Betti number} \quad  \beta_i(X) \in \bZ$$
and 
$$\emph {the virtual Poincar\'e polynomial} \quad \beta_t(X) = \sum_i\beta_i(X)t^i \in \bZ[t].$$They are both identical to the usual Betti number and Poincar\'e polynomial on compact nonsingular varieties. For the above two figure eight curves $F8$ and $tF8$ we indeed have that
$$\beta_t(F8) \not = \beta_t(tF8).$$
For more details, see \cite{MP} and \cite{To2}, and see also Remark \ref{beta-remark}. 
\end{rem}

Finally, in passing, we also mention the following ``cobordism" counting $c_{cob}$ on the category of closed oriented differential manifolds or the category of stably almost complex manifolds :
\begin{enumerate}
\item
$ X\cong X'$ (cobordant, or bordant) $\Longrightarrow$ $c_{cob}(X) = c_{cob}(X')$,

\item
$c_{cob}(X \sqcup Y) = c_{cob}(X) + c_{cob}(Y)$ \quad 
(Note: in this case $c_{cob}(X \setminus Y)$ does not make  a sense, because $X \setminus Y$ has to be a closed oriented manifold)

\item
$c_{cob}(X \times Y) = c_{cob}(X)\cdot c_{cob}(Y)$,

\item
$c_{cob}(pt) = 1.$ \\
\end{enumerate} 

As in the cases of the previous countings,  (1) and (3) imply that $c_{cob}(pt) = 0$ or $1$. It follows from (3) that $c_{cob}(pt) = 0$ implies that $c_{cob}(X) = 0$ for any closed oriented differential manifolds $X$. Thus the last condition $c_{cob}(pt) = 1$ means that our $c_{cob}$ is a nontrivial one. Such a ``cobordism" counting $c_{cob}$ is nothing but a \underline {genus} such as \underline {signature, $\hat A$-genus, elliptic genus}. As in Hirzebruch's book, a genus is usually defined as a nontrivial one satisfying the above three properties (1), (2) and (3). Thus, it is the same as the one given above.\\

Here is a very simple problem on genera of closed oriented differentiable manifolds or stably almost complex manifolds:

\begin{prob} \label {all genera} Determine all genera.
\end{prob}

Let $\op {Iso}(G)_{n}$ be the set of isomorphism classes of smooth closed
(and oriented) pure $n$-dimensional manifolds $M$ for $G=O$ (or $G=SO$), 
or of pure $n$-dimensional weakly (``$=$ stably") almost complex manifolds $M$ for $G=U$,
i.e. $TM\oplus \bR_{M}^N$ is a complex vector bundle (for suitable $N$, with $\bR_{M}$ the trivial real line bundle over $M$). Then
$$\op {Iso}(G):= \bigoplus_{n}\: \op {Iso}(G)_{n}$$
becomes a commutative graded semiring with addition and multiplication given by disjoint union
and exterior product, with $0$ and $1$ given by the classes of the empty set and one point
space. 

Let $\Omega^{G}:=\op {Iso}(G)/\sim$ be the corresponding {\em cobordism ring}
of closed ($G=O$) and oriented ($G=SO$) or weakly (``$=$ stably") almost complex manifolds ($G=U$)
as dicussed for example in \cite{Stong}.
Here $M\sim 0$ for a closed pure $n$-dimensional $G$-manifold $M$ if and only if 
there is a compact pure $n+1$-dimensional $G$-manifold $B$ with boundary $\partial B\simeq M$.
Note that this is indeed a ring with $-[M]=[M]$ for $G=O$ or $-[M]=[-M]$ for $G=SO,U$,
where $-M$ has the opposite orientation of $M$. Moreover, for $B$ as above with
$\partial B\simeq M$ one has 
$$TB|\partial B \simeq TM\oplus \bR_{M}.$$
This also explains the use of the stable tangent bundle for the definition of a
stably or weakly almost complex manifold.

The following structure theorems are fundamental:(see \cite[Theorems on p.177 and p.110]{Stong}):

\begin{thm}\label{cob-structure} 
\begin{enumerate}
\item {\em (Thom)} $\Omega^{SO}\otimes \bQ = \bQ[\bP^2, \bP^4, \bP^6, \cdots, \bP^{2n} \dots ]$
is a polynomial algebra in the classes of the complex even dimensional projective spaces.
\item {\em (Milnor)} $\Omega^{U}_{*}\otimes \bQ = \bQ[\bP^1, \bP^2, \bP^3, \cdots, \bP^{n} \dots ]$
is a polynomial algebra in the classes of the complex  projective spaces.
\end{enumerate}
\end{thm}


So, if we consider a commutative ring $R$ without torsion for a genus $\ga: \Omega^{SO} \to R$, then the genus $\ga$ is completely determined by the value $\ga(\bP^{2n})$ of the cobordism class of each even dimensional complex projective sapce $\bP^{2n}$. Then using this value one could consider its generating ``function" or formal power series such as $\sum_n \ga(\bP^{2n})x^n$, or $\sum_n \ga(\bP^{2n})x^{2n}$, and etc. In fact, a more interesting problem is to determin all \emph {rigid} genera such as the above mentioned signature $\sigma$ and $\hat A$; namely a genera satisfying the following multiplicativity property stronger than the product property (3):\\
\begin{enumerate}
\item[(3)$_{rigid}$]: $\ga(M) = \ga(F) \ga(B)$ for a fiber bundle $M \to B$ with its fiber $F$ and compact connected structure group.\\
\end{enumerate}

\begin{thm} Let $\log_{\ga}(x)$ be the following ``logarithmic" formal power series in $R[[x]]$:
$$\log_{\ga}(x) := \sum _n \frac {1}{2n+1}\ga(\bP^{2n})x^{2n+1}.$$
Then the genus $\ga$ is rigid if and only if it is an elliptic genus, i.e., its logarithm $\log_{\ga}$ is an elliptic integral, i.e.,
$$\log_{\ga}(x) = \int_0^x \frac {1}{\sqrt {1 -2 \delta t^2 + \epsilon t^4 }} dt$$
for some $\delta, \epsilon \in R$.
\end{thm}
 The ``only if" part was proved by S. Ochanine \cite{Ochanine} and  the ``if part" was first ``physically" proved by E. Witten \cite{Witten} and later ``mathematically" proved by C. Taubes \cite {Taubes} and also by R. Bott and C. Taubes \cite{Bott-Taubes}. See also B. Totaro's papers \cite{To1, To3}.

$$ \text {\bf cobordism counting $c_{cob}$ : Thom's Theorem}$$

$$ \text {\bf rigid genus = elliptic genus : elliptic integral}$$\


The above oriented cobordism group $\Omega^{SO}$ was extended by M. Atiyah \cite{Atiyah} to a generalized cohomology theory, i.e., the (oriented) cobordism theory $MSO^*(X)$ of a topological space $X$. The theory $MSO^*(X)$ is defined by the so-called Thom spectra, i.e, the infinite sequence  of Thom complexes $MSO(n)$: for a topological pair $(X,Y)$ with $Y \subset X$
$$MSO^k(X, Y):= \lim _{n \to \infty} [\Sigma^{n-k}(X/Y), MSO(n) ].$$
Here the homotopy group $[\Sigma^{n-k}(X/Y), MSO(n) ]$ is stable. \\

As a covariant or homology-like version of $MSO^*(X)$, M. Atiyah \cite{Atiyah} introduced the bordism theory $MSO_*(X)$ geometrically in a quite simple manner: Let $f_1:M_1 \to X$, $f_2:M_2 \to X$ be continuous maps from closed oriented $n$-dimensional manifolds to a topological space $X$. $f$ and $g$ are said to be bordant if there exists an oriented manifold $W$ with boundary and a continuous map $g:W \to X$ such that 

(1) $g|_{M_1} = f_1$ and $g|_{M_2} = f_2$, 

(2) $\partial W = M_1 \cup - M_2$ , where $- M_2$ is $M_2$ with its reverse orientation. 

\noindent It turns out that $MSO_*(X)$ is a generalized homology theory and
$$MSO^{0}(pt) = MSO_0(pt) = \Omega^{SO}.$$
M. Atiyha \cite{Atiyah} also showed the Poincar\'e duality for an oriented closed manifold $M$ of dimension $n$:
$$MSO^k (M) \cong MSO_{n-k}(M).$$

If we replace $SO(n)$ by the other groups $O(n)$, $U(n)$, $Spin(n)$, we get the corresponding cobordism and bordism theories.\\

\begin{rem} (Elliptic Cohomology)
Given a ring homomorphism $\varphi: MSO^*(pt) \to R$, $R$ is an $MSO^*(pt) $-module and 
$$MSO^*(X) \otimes _{MSO^*(pt) } R$$
becomes ``almost" a generalized cohomology theory, namely it does not necessarily satisfy the Exactness Axiom.
P. S. Landweber \cite{Landweber0} gave an algebraic criterion (called the Exact Functor Theorem) for it to become a generalized cohomology theory. Applying Landweber's Exact Functor Theorem, P. E. Landweber, D. C. Ravenel and R. E. Stong \cite{LRS} showed the following theorem:

\begin{thm} For the elliptic genus $\ga: MSO^*(pt) = MSO_*(pt) = \Omega \to \bZ[\frac {1}{2}][\delta, \epsilon]$, the following functors are generalized cohomology theories:
$$MSO^*(X) \otimes _{MSO^*(pt)}\bZ[\frac {1}{2}][\delta, \epsilon][\epsilon^{-1}], \hspace{0.9cm}$$
$$MSO^*(X) \otimes _{MSO^*(pt)}\bZ[\frac {1}{2}][\delta, \epsilon][(\delta^2 - \epsilon)^{-1}],$$
$$MSO^*(X) \otimes _{MSO^*(pt)}\bZ[\frac {1}{2}][\delta, \epsilon][\Delta^{-1}], \hspace{0.9cm}$$
where $\Delta = \epsilon (\delta^2 - \epsilon)^2.$
\end{thm}
 More generally J. Franke \cite{Franke} showed the following theorem:
 \begin{thm} For the elliptic genus $\ga: MSO^*(pt) = MSO_*(pt) = \Omega^{SO} \to \bZ[\frac {1}{2}][\delta, \epsilon]$, the following functor is a generalized cohomology theory:
$$MSO^*(X) \otimes _{MSO^*(pt)} \bZ[\frac {1}{2}][\delta, \epsilon][P(\delta, \epsilon) ^{-1}],$$
where $P(\delta, \epsilon)$ is a homogeneous polynomial of positive degree with $\op{deg} \delta = 4, \op{deg} \epsilon = 8$.
\end{thm}

The generalized cohomology theory 
$$MSO^*(X) \otimes _{MSO^*(pt)}\bZ[\frac {1}{2}][\delta, \epsilon][P(\delta, \epsilon) ^{-1}]$$
 is called an \emph {elliptic cohomology theory} (for a recent survey of it see J. Lurie's paper \cite{Lu}). It is defined in an algebraic manner,
but not in a more topologically or geometrically simpler manner as in K-theory or the bordism theory $MSO_*(X)$.
So, people have been searching for a reasonable geometric or
topological construction of the elliptic cohomology (cf. \cite{KrSt}).\\
\end{rem}

\begin{rem} (Just a mumbo jumbo) In the above we see that if you just count points of a variety simply as a set, we get an infinity unless it is a finite set or the trivial one $0$, but that if we count it ``respecting" the topological and algebraic structures you get a certain reasonable number which is not an infinity. Getting carried away, the ``zeta function-theoretic" formulae such as
$$1 + 1 + 1 + \cdots + 1 + \cdots = - \frac{1}{2} = \zeta(0)$$
$$1 + 2 + 3 + \cdots + n + \cdots = - \frac{1}{12} = \zeta(-1)$$
$$1 + 2 ^2+ 3^2 + \cdots + n^2 + \cdots = 0 = \zeta(-2)$$
$$1^3 + 2^3 + 3^3 +\cdots + n^3 + \cdots = \frac{1}{120} = \zeta(-3)$$
$$\cdots$$
could be considered as some kind of  counting an inifite set ``respecting" some kind of ``zeta-structure" on it, whatever the zeta-structure is.  In nature, the above equality $1^3 + 2^3 + 3^3 +\cdots + n^3 + \cdots = \frac{1}{120}$ is explained as the \emph {Casimir Effect} (after Dutch physicists Hendrik B. G. Casimir). So,  nature perhaps already knows what the ``zeta-structure" is. It would be fun (even non-mathematically) to wonder or imagine what would be a ``zeta-structure" on the natural numbers $\bN$, or the integers $\bZ$ or the rational numbers $\bQ$, or more generally ``zeta-structured" spaces or varieties.  Note that, as the topological counting $c_{top} = \chi$ was found by Euler, the ``zeta-theoretical counting" (denoted by $c_{zeta}$ here)  was also found by Euler ! \\
\end{rem}

\section {Motivic Characteristic Classes}

Any ``algebraic counting" $c_{alg}$ gives rise to the following na\" \i ve ring homomorphism to a commutative ring $R$ with unit $1$ :

$$c_{alg} : Iso (\Cal {VAR}) \to R \quad \text {defined by} \quad c_{alg}([X]) : = c_{alg}(X).$$ \

Here $Iso (\Cal {VAR}) $ is the free abelian group generated by the isomorphism classes $[X]$ of complex varieties. The following additivity relation 
$$\text {$c_{alg}([X]) = c_{alg}([X \setminus Y]) + c_{alg}([Y])$ for any closed subvariety $Y \subset X$}$$
in other words,
$$\text {$c_{alg}([X] -[Y] - [X \setminus Y]) = 0 $ for any closed subvariety $Y \subset X$ }$$
induces the following finer ring homomorphism:\\
$$c_{alg} : K_0(\Cal {VAR}) \to R \quad \text {defined by} \quad c_{alg}([X]) : = c_{alg}(X).$$ \\ 
Here $K_0(\Cal {VAR})$ is the Grothendieck ring of complex algebraic varieties, i.e., $Iso (\Cal {VAR})$ modulo the following additivity relation 
$$ \text {$[X] = [X \setminus Y] + [Y]$ for any closed subvariety $Y \subset X$}$$
or, in other words, $Iso (\Cal {VAR})$  mod out the subgroup generated by the elements of the form \\
$$[X] - [Y] - [X \setminus Y]$$ 
for any closed subvariety $Y \subset X$.

The equivalence class of $[X]$ in $K_0(\Cal {VAR})$ should be written as, say $[[X]]$, but we just use the symbol $[X]$ for the sake  of simplicity. 

More generally, let $y$ be an indeterminate and we can consider the following homomorphism $c_{alg} := \chi_y := \chi_{y, -1}$, i.e.,  
$$c_{alg} : K_0(\Cal {VAR}) \to \bZ[y] \quad \text {with} \quad c_{alg}(\bC^1) = -y.$$\
This shall be called \emph{\bf a motivic characteristic}, to emphasize the fact that its domain is the Grothendieck ring of varieties.\\

\begin{rem} In fact, for the category $\Cal {VAR}(k)$ of algebraic varieties over any field, the above Grothendieck ring $K_0(\Cal {VAR}(k))$ can be defined in the same way. 
\end{rem}

What we want to do is an analogue to the way that Grothendieck extended the celebrated Hirzebruch--Riemann--Roch Theorem (which was the very beginning of the Atiyah--Singer Index Theorem) to Grothendieck--Riemann--Roch Theorem. Namely we want to solve the following problem:

\begin{prob}\label{problem1} Let $R$ be a commutative ring with unit $1$ such that $\bZ \subset R$, and let $y$ be an indeterminate and Do there exist some covariant functor $\spadesuit$ and some natural transformation (here pushforwards are considered for proper maps)
$$\natural : \spadesuit(\quad)  \to  H_*^{BM}(\quad) \otimes R[y]$$
 such that
 \begin{enumerate}
\item
$\spadesuit(pt) = K_0(\Cal {VAR})$ ,

\item 
$\natural (pt) = c_{alg}$, i.e., 
$$\natural (pt) = c_{alg} : \spadesuit(pt) = K_0(\Cal {VAR})  \to  R[y] = H_*^{BM}(pt)\otimes R[y].$$

\item
For the mapping $\pi_X: X \to pt$ to a point, for a certain distinguished element $\Delta_X \in \spadesuit(X)$
we have
$$ {\pi_X}_* (\natural(\Delta_X)) = c_{alg}(X) \in R[y] \quad \text {and} \quad {\pi_X}_*(\Delta_X) = [X] \in K_0(\Cal {VAR})  \quad ?$$
$$\CD
\spadesuit(X) @> \natural(X)  >>  H_*^{BM}(X) \otimes R[y] \\
@V {{\pi_X}_*} VV @VV {{\pi_X}_*} V\\
 \spadesuit(pt) = K_0(\Cal {VAR}) @>> {\natural(pt) = c_{alg}} > R[y] .  \endCD
$$
(If there exists such one, then $\natural(\Delta_X)$ could be called the {\bf motivic characteristic class} corresponding to the motivic characteristic $c_{alg}(X)$, just like the Poincar\'e dual of the total Chern cohomology class $c(X)$ of a complex manifold $X$ corresponds to the Euler--Poincar\'e characteristic $\chi(X)$, i.e., ${\pi_X}_*(c(X) \cap [X]) = \chi(X)$.)
\end{enumerate}
\end{prob}

A more concrete one for the Hodge--Deligne polynomial (a prototype of this problem was considered  in \cite{Yokura-hodge} (cf. \cite{Y4}):
\begin{prob}\label{problem2}  Let $R$ be a commutative ring with unit $1$ such that $\bZ \subset R$ , and let $u, v$ be two indeterminates. Do there exist some covariant functor $\spadesuit$ and some natural transformation  (here pushforwards are considered for proper maps)
$$\natural : \spadesuit(\quad)  \to  H_*{BM}(\quad) \otimes R[u,v]$$
 such that
 \begin{enumerate}
\item
$\spadesuit(pt) = K_0(\Cal {VAR})$ ,

\item 
$\natural (pt) = \chi _{u,v}$, i.e., 
$$\natural (pt) = \chi _{u,v} : \spadesuit(pt) = K_0(\Cal {VAR})  \to  R[u, v] = H_*^{BM}(pt)\otimes R[u,v] .$$

\item
For the mapping $\pi_X: X \to pt$ to a point, for a certain distinguished element $\Delta_X \in \spadesuit(X)$
we have
$$ {\pi_X}_* (\sharp(\Delta_X)) = \chi_{u,v}(X) \in R[u,v] \quad \text {and} \quad {\pi_X}_*(\Delta_X) = [X] \in K_0(\Cal {VAR})  \quad ?$$\\
\end{enumerate}
\end{prob}

One reasonable candidate for the covariant functor $\spadesuit$ is the following:
\begin{defn} (e.g., see \cite {Lo}) \emph{The relative Grothendieck group of $X$}, denoted by
$$K_0(\Cal {VAR}/X)$$
is defined to be the free abelian group $Iso(\Cal {VAR}/X)$ generated by the isomorphism classes $[V  \xrightarrow{h}  X]$ of morphsim of complex algebraic varieties over $X$, $h:V \to X$,  modulo the following additivity relation
$$ \text {$[V  \xrightarrow{h}  X] = [V \setminus Z  \xrightarrow{{h_|}_{V \setminus Z}}  X + [Z \xrightarrow{{h_|}_Z}  X]$ for any closed subvariety $Z \subset V$},$$
namely, $Iso(\Cal {VAR}/X)$ modulo the subgroup generated by the elements of the form
$$[V  \xrightarrow{h}  X] - [Z \xrightarrow{{h_|}_Z}  X] - [V \setminus Z  \xrightarrow{{h_|}_{V \setminus Z}}  X] $$ 
for any closed subvariety $Z \subset V$. 
\end{defn}

\begin{rem} For the category $\Cal {VAR}(k)$ of algebraic varieties over any field, we can consider the same relaive Grothendieck ring $K_0(\Cal {VAR}(k)/X).$ 
\end{rem} 

\noindent
{\bf NOTE 1}: $K_0(\Cal {VAR}/pt) = K_0(\Cal {VAR})$ 

\noindent
{\bf NOTE 2}: $K_0(\Cal {VAR}/X)$\footnote{According to a recent paper by M. Kontsevich (``Notes on motives in finite characteristic", math.AG/ 0702206), Vladimir Drinfeld calls an element of $K_0(\Cal {VAR}/X)$ ``poor man's motivic function"} is a covariant functor with the obvious pushforward: for a morphism 
$f : X \to Y$, the pushforward
$$f_*: K_0(\Cal {VAR}/X) \to K_0(\Cal {VAR}/Y)$$
is defined by
$$f_*([V  \xrightarrow{h}  X]) := [V  \xrightarrow{f \circ h}  Y].$$

\noindent
{\bf NOTE 3}: Although we do not need the ring structure  on $K_0(\Cal {VAR}/X)$ in later discussion, the fiber product gives a ring structure on it:
$$[V_1  \xrightarrow{h_1}  X] \cdot [V_2  \xrightarrow{h_2}  X] := [V_1 \times _X V_2  \xrightarrow{h_1 \times _X h_2}  X]$$ 

\noindent
{\bf NOTE 4}: If $\spadesuit(X) = K_0(\Cal {VAR}/X) $, then the distinguished element $\Delta_X$ is the isomorphism class of the identity map:
$$\Delta_X = [X  \xrightarrow{\op {id}_X}  X].$$

If we impose one more requirement in the above Problem \ref{problem1} and  Problem \ref{problem2} , we can solve the problem. The additional one is the following \emph{normalization condition} or \emph{``smooth condition"} that for nonsingular $X$
$$\natural (\Delta_X) = c\ell(TX) \cap [X]$$
for a certain ``normalized" multiplicative characteristic class $c\ell$ of complex vector bundles. Note that $c\ell$ is a polynomial in the Chern classes such that it satisfies the normalization that. Here ``normalized" means that $c\ell(E) = 1$ for any trivial bundle $E$ and ``multiplicative" means that $c\ell(E \oplus F) = c\ell(E)c\ell(F)$,  which is called the \emph{Whitney sum formula}. As to the Whitney sum formula, in the analytic or algebraic context, one askes for this multiplicativity for a short exact sequence of vector bundles (which splits only in the topological context):
$$c\ell(E) = c\ell(E')c\ell(E'') \quad \text {for } \quad 1 \to E' \to E \to E'' \to 1.$$The above extra requirement of ``smooth condition" turns out to be a quite natural one in the sense that the other well-known/studied characteristic homology classes of possibly singular varieties are formulated as natural transformations satisying such a normalization condition, as recalled later. Also, as made in Conjecture \ref {conjecture1} in a later section, this seemingly strong requirement of normalization condition could be eventually dropped. \\ 

\begin{ob} Let $\pi_X: X \to pt$ be the mapping to a point. Then it follows from the naturality of $\natural$ and the above normalization condition  that for a nonsingular variety $X$ we have
\begin{align*}
 c_{alg}([X]) &= \natural ({\pi_X}_* ([X  \xrightarrow{\op {id}_X}  X])) \\
 & = {\pi_X}_*(\natural ([X  \xrightarrow{\op {id}_X}  X]))\\
 & = {\pi_X}_* (c\ell(TX) \cap [X]).
 \end{align*}
Therefore the above normaization condition on nonsingular varieties implies that for a nonsingular variety $X$ the ``algebraic counting" $c_{alg}(X)$ has to be a characteristic number or Chern number \cite{Ful, MiSt}. Thus it is another requirement on $c_{alg}$, but it is an inevitable one if we want to capture it functorially (i.e., like a Grothendieck--Riemann--Roch type) together with the above normalization condition for smooth varieties.\\
\end{ob}

Furthermore, this normalization condition turns out to be quite essential and in fact it automatically determines the characteristic class $c\ell$ as follows, if we consider the bigger ring $\bQ[y]$ instead of $\bZ[y]$:

\begin{pro} If the above normalization condition is imposed in the above problems, then the multiplicative characteristic class $c\ell$ with coefficients in $\bQ[y]$ has to be the generalized Todd class, or the Hirzebruch class $T_y$, i.e., for a complex vector bundle $V$
$$T_y(V) := \prod _{i =1}^{rank V}\left (\frac {\alpha_i (1+y)}{1 - e^{-\alp_i(1+y)}} - \alp_i y \right )$$
with $\alp_i$ are the Chern roots of the vector bundle, i.e., $\displaystyle c(V) = \prod _{i =1}^{rank V}(1 + \alp_i)$.
\end{pro}

\begin{proof} We note that the multiplicativity of $c\ell$ gurantees that for two smooth compact varieties $X$ and $Y$, we have
$${\pi_{X \times Y}}_*(c\ell(T(X \times Y) \cap [X \times Y]) = {\pi_X}_*(c\ell(TX) \cap [X]) \cdot {\pi_Y}_*(c\ell(TY) \cap [Y]),$$
i.e., the Chern number is multiplicative, i.e., it is compatible with the multiplicativity of $c_{alg}$. Now Hirzebruch's theorem \cite[Theorem 10.3.1]{Hi} 
says that if the multiplicative Chern number defined by a multiplicative characteristic class $c\ell$ with coefficients in $\bQ[y]$ satisfies that the corresponding characteristic number of the complex projective space $\bP^n$ is equal to $1 -y +y^2 -y^3 + \cdots + (-y)^n$, then the multiplicative characteristic class $c\ell$ has to be the generalized Todd class, i.e., the Hirzebruch class $T_y$ above. \\
\end{proof}

\begin{rem}
 In other words, in a sense $c_{alg}(\bC^1)$ uniquely determines the class version of the motivic characteristic $c_{alg}$, i.e., the motivic characteristic class. This is very similar to the fact foreseen that $c_{top}(\bR^1) = -1$ uniquely determines the ``topological counting" $c_{top}$.
\end{rem}
\noindent
{\bf IMPORTANT NOTE}: This Hirzebruch class $T_y$ specializes to the following important characteristic classes:\

\hspace {1cm} $\displaystyle  y = -1: T_{-1}(V) = c(V) = \prod _{i =1}^{rank V}(1 + \alp_i)$ the total Chern class

\hspace {1cm} $\displaystyle y= 0: \quad T_0(X) = td(V) = \prod _{i =1}^{rank V}\frac {\alpha_i }{1 - e^{-\alp_i}} $ the total Todd class

\hspace {1cm} $\displaystyle y= 1: \quad T_{1}(X) = L(V) = \prod _{i =1}^{rank V}\frac {\alpha_i }{\tanh \alp_i} $ the total Thom--Hirzebruch class. \\

Now we are ready to state our answer for Problem \ref{problem1}, which is one of the main theorems of \cite{BSY1}:

\begin{thm} \label{theorem1} (Motivic Characteristic Classes) Let $y$ be an indeterminate. 
\begin{enumerate}
\item
There exists a unique natural transformation
$${T_y}_*:  K_0(\Cal {VAR}/X) \to H_*^{BM}(X)\otimes \bQ[y]$$
satisfying the normalization condition that for a nonsingular variety $X$
$${T_y}_*([X \xrightarrow{id_{X}}  X]) =  T_y(TX) \cap [X].$$

\item 
For $X = pt$,
$${T_y}_*: K_0(\Cal {VAR}) \to \bQ[y]$$
is equal to the Hodge--Deligne polynomial  
 $$\chi_{y, -1}:K_0(\Cal {VAR}) \to  \bZ[y] \subset \bQ[y].$$
 namely,
$${T_y} _*([X \to pt]) = \chi_{y, -1}([X]) = \sum_{i, p\geq 0} (-1)^i \op {dim}_{\bC} (Gr^p_F H^i_c(X, \bC)) (-y)^p.$$
$\chi_{y, -1}(X)$ is simply denoted by $\chi_y(X)$.
\end{enumerate}
\end{thm}

\begin{proof} (1) : The main part is of course the existence of such a ${T_y}_*$, the proof of which is outlined in a later section. Here we point out only the uniqueness of ${T_y} _*$, which follows from resolution of singularities. More precisely it follows from
\begin{enumerate}

\item [(i)]
Nagata's compactification theorem, or the projective closure of affine subvarieties if we do not use the fancy Nagata's compactification theorem: We get the following surjective homomorphism
$$ A: Iso^{prop}(\Cal {VAR}/X) \twoheadrightarrow K_0(\Cal {VAR}/X) $$
where $ Iso^{prop}(\Cal {VAR}/X) $ is the free abelian group generated by the isomorphism class of \emph{proper }morphisms to $X$.

\item[(ii)]
Hironaka's resolution of singularities (i.e., we can show by the resolution of singularities and by the induction on dimension that any isomorphism class $[Y \xrightarrow{h}  X]$ can  be expressed as 
$$\sum _{V} a_V [V \xrightarrow{h_V}  X]$$
with $V$ being nonsingular and $h_V : V \to X$ being proper.): Here we get the following surjective maps

$$ Iso^{prop}(\Cal {SM}/X) \twoheadrightarrow  Iso^{prop}(\Cal {VAR}/X), $$
therefore
$$B:  Iso^{prop}(\Cal {SM}/X) \twoheadrightarrow K_0(\Cal {VAR}/X). $$
Here $ Iso^{prop}(\Cal {SM}/X) $ is the the free abelian group generated by the isomorphism class of \emph{proper }morphisms from \emph{smooth varieties} to $X$.

\item[(iii)]
The above normalization condition (or the ``smooth" condition).

\item[(iv)] The naturality of ${T_y} _*$ .

\end{enumerate}
 The above two surjective homomorphisms $A$ and $B$ also play some key roles in the proof of the existence of ${T_y}_*$.
 
 (2):  As pointed out in (ii), $K_0(\Cal {VAR})$ is generated by the isomorphism classes of compact smooth varieties.  On a nonsingular compact variety $X$ we have

$$\chi_{y,-1}(X) = \sum_{p,q \geq 0} (-1)^q \op {dim}_{\bC} H^q(X; \Omega_X^p) y^p,$$
which is denoted by $\chi_y(X)$ and is called the Hirzebruch's  $\chi_y$-genus.
Next we have the following {\bf generalized Hirzebruch--Riemann--Roch Theorem} (abbr. , gHRR) \cite{Hi}
$$\chi_y(X) = \int _X T_y(TX) \cap [X].$$ 
Since $\displaystyle \int _X T_y(TX) \cap [X] = {\pi_X}_*(T_y(TX) \cap [X]) = {T_y}_*([X \to pt])$, we have the following equality on the generators of $K_0(\Cal {VAR})$
 $${T_y}_*([X \to pt]) = \chi_{y,-1}([X])$$ 
 thus on the whole $K_0(\Cal {VAR})$ we must have that ${T_y}_*= \chi_{y,-1}$.
\end{proof}

\begin{rem} \label {elliptic} Problem \ref{problem2} is a problem slightly more general than Problem \ref{problem1} in the sense that it invloves two indeterminates $u, v$. However, the whole important keys are the normalization condition for smooth compact varieties and the fact that $\chi_{u,v}(\bP^1) = 1 +uv + (uv)^2 + \cdots + (uv)^n$, which automatically implies that $c\ell = T_{-uv}$, as shown in the above proof. In fact, we can say more about $u, v$; in fact either $u= -1$ or $v = -1$, as shown below (see also \cite{Jo}(math.AG/0403305v4, not a published version)). Hence, we can conclude that for Problem \ref {problem2} there is NO such transformation $\sharp: K_0(\Cal {VAR}/-) \to H_*^{BM}(-)\otimes R[u,v]$ with \underline{both intermediates $u$ and $v$ varying}:

For a smooth $X$, suppose that for a certain multiplicative characteristic class $c\ell$
$$\chi_{u,v}(X) = {\pi_X}_*(c\ell(TX) \cap [X]).$$
In particular, let us consider a smooth elliptic curve $E$ and consider any $d$-fold covering 
$$\pi: \widetilde E \to E$$
with $\widetilde E$ being a smooth elliptic curve. Note that 
$$T\widetilde E = \pi^*TE,$$
$$\chi_{u,v}(E) = \chi_{u,v}(\tilde E) = 1 + u + v + uv = (1+u)(1+v).$$
Hence we have 
\begin{align*}
(1+u)(1+v) & = \chi_{u,v}(\widetilde E) \\
& = {\pi_{\widetilde E}}_*(c\ell(T\widetilde E) \cap [\widetilde E]) \\
& = {\pi_{\widetilde E}}_*(c\ell(\pi^*TE) \cap [\widetilde E]) \\
& = {\pi_E}_*\pi_*(c\ell(\pi^*TE) \cap [\widetilde E]) \\
& = {\pi_E}_*(c\ell(TE) \cap \pi_*[\widetilde E]) \\
& = {\pi_E}_*(c\ell(TE) \cap d[ E]) \\
& = d \cdot {\pi_E}_*(c\ell(TE) \cap [ E]) \\
& = d \cdot \chi_{u,v}(E)\\
& = d (1+u)(1+v). 
\end{align*}
Thus we get that $(1+u)(1+v) = d(1+u)(1+v)$. Since $d \not = 0$, we must have that 
$$(1+u)(1+v) = 0, i.e.,  u = -1 \quad \text {or} \quad v = -1.$$
\end{rem} 

\begin{rem} Note that $\chi_{u,v}(X)$ is symmetric in $u$ and $v$, thus both special cases $u=-1$ and $v = -1$ give rise to the same $c\ell = T_y$. It suffices to check it for a compact nonsingular variety $X$. In fact this follows from the Serre duality.
\end{rem}
\begin{rem}
The heart of the mixed Hodge structure is certainly the existence of the weight filtration $W^{\bullet}$ and the Hodge--Deligne polynomial, i.e., the algebraic counting $c_{alg}$, involves the mixed Hodge structure, i.e., both the weight filtration $W^{\bullet}$ and the Hodge filtration $F_{\bullet}$ . However, when one tries to capture $c_{alg}$ \emph{functorially}, then only the Hodge filtration $F_{\bullet}$ gets involved and the weight filtration \emph{does not}, as seen in the Hodge genus $\chi_y$.
\end{rem}

\begin{defn} For a possibly singular variety $X$\\
$$ {T_y}_*(X) := {T_y}_*([X \xrightarrow{id_{X}}  X]) $$
shall be called the \emph {Hirzebruch class of $X$}.\\
\end{defn}

\begin{cor} The degree of the $0$-dimensional component of the Hirzebruch class of a compact complex algebraic variety $X$ is nothing but the Hodge genus:
$$ \chi_y(X) = \int _X {T_y}_*(X).$$\
\end{cor}

This is another singular analogue of the above Generalized Hirzebruch--Riemann--Roch Theorem ``$\chi_y = T_y$", which is a generalization of the famous Hirzeburch's Riemman--Roch Theorem (which was further generalized to the Grothendieck--Riemann--Roch Theorem) 
$$\text {Hirzebruch--Riemann--Roch:} \quad  p_a(X) = \int _X td(TX) \cap [X]$$
with $p_a(X)$ the arithmetic genus and $td(V)$ the original Todd class.
Noticing the above specializations of $\chi_y$ and $T_y(V)$, this gHRR is a unification of the following three well-known theorems:\\

\hspace {1cm} $y = -1$: \quad The Gauss--Bonnet Theorem (or Poincar\'e --Hopf Theorem) :
$$\chi(X) = \int_X c(X) \cap [X]$$

\hspace {1cm} $y = 0$: \quad The Hirzebruch--Riemann--Roch : 
$$p_a(X)= \int_X td(X) \cap [X]$$

\hspace {1cm} $y = 1$: \quad The Hirzebruch's Signature Theorem: 
$$\sigma(X) = \int_X L(X) \cap [X].$$\

\section {Proofs of the existence of the motivic characteristic class ${T_y}_*$} 

Our motivic characteristic class transformation

$${T_y} _*: K_0(\Cal {VAR}/X) \to H_*^{BM}(X)\otimes \bQ[y] $$
is obtained as the composite 
$${T_y} _* = \widetilde {td_{*(y)}^{BFM}} \circ \Lambda _y^{mot}$$
of the following natural transformations:

$$\Lambda _y^{mot}: K_0(\Cal {VAR}/X) \to G_0(X) \otimes \bZ[y] \qquad \qquad $$

and 
$$\widetilde {td_{*(y)}^{BFM}}: G_0(X) \otimes \bZ[y] \to H_*^{BM}(X) \otimes \bQ[y, (1+y)^{-1}] .$$

Here, in order to describe $\widetilde {td_{*(y)}^{BFM}}$,  we need to recall the following Baum--Fulton--MacPherson's Riemann--Roch or Todd class for singular varieties \cite{BFM}:

\begin{thm}
There exists a unique natural transformation
$$td_*^{BFM}:G_0(-) \to H_*^{BM}(-)\otimes \bQ$$
such that for a smooth $X$
$$td_*^{BFM}(\Cal O_X) = td(TX) \cap [X].$$
Here $G_0(X)$ is the Grothendieck group of coherent sheaves on $X$, which is a covariant functor with the pushforward $f_*:G_0(X) \to G_0(Y)$ for a proper morhphism $f:X \to Y$ defined by
$$f_!(\Cal F) =  \sum_j (-1)^j [R^j f_*\Cal F].$$ 
\end{thm}

Let us set 
$$td_*^{BFM}(X) := td_*^{BFM}(\Cal O_X),$$
which shall be called the Baum--Fulton--MacPherson Todd class of $X$. Then we have
$$p_a(X) = \chi(X, \Cal O_X) = \int_X td_*^{BFM}(X)\quad \text {(HRR-type theorem)}.$$

Let 
$$td_{*i}^{BFM}: G_0(X)  \xrightarrow {td_*^{BFM}}  H_*^{BM}(X) \otimes \bQ \xrightarrow {projection}  H_{2i}^{BM}(X) \otimes \bQ$$
be the $i$-th (i.e., $2i$-dimensional) component of $td_*^{BFM}$. Then the above \emph{twisted BFM-Todd class transformation} or \emph {twisted BFM-RR transformation} (cf. \cite{Y3})

$$\widetilde {td_{*(y)}^{BFM}}: G_0(X) \otimes \bZ[y] \to H_*^{BM}(X) \otimes \bQ[y, (1+y)^{-1}]$$
is defined by
$$\widetilde {td_{*(y)}^{BFM}}: = \sum_{i\geq 0} \frac{1}{(1+y)^i} td_{*i}^{BFM}.$$

$\Lambda _y^{mot}: K_0(\Cal {VAR}/X) \to G_0(X) \otimes \bZ[y]$  is the main key and in our paper \cite{BSY1} it is denoted by $mC_*$ and called the \emph{motivic Chern class}. In this paper, we use the above symbol to emphasize the property of it:

\begin{thm} \label{lambda}(``motivic" $\lambda_y$-class transformation)
There exists a unique natural transformation
$$\Lambda _y^{mot}: K_0(\Cal {VAR}/X) \to G_0(X) \otimes \bZ[y]$$
satisfying the normalization condition that for smooth $X$
$$\Lambda _y^{mot}([X \xrightarrow {\op {id}} X]) = \sum_{p = 0}^{dim X} [\Omega_X^p]y^p = \lambda_y(T^*X) \otimes [\Cal O_X].$$
Here $\displaystyle \lambda _y(T^*X) = \sum_{p= 0}^{dim X} [\Lambda ^p (T^*X)]y^p$ and $\otimes [\Cal O_X] : K^0(X) \cong G_0(X)$ is the isomorphism for smooth $X$, i.e., taking the sheaf of local sections.
\end{thm}

\begin{thm} \label {Ty*} The natural transformation 
$${T_y} _* := \widetilde {td_{*(y)}^{BFM}} \circ \Lambda _y^{mot}:  K_0(\Cal {VAR}/X)  \to H_*^{BM}(X) \otimes \bQ[y] \subset  H_*(X) \otimes \bQ[y, (1+y)^{-1}] $$
satisfies the normalization condition that for smooth $X$
$${T_y} _* ([X \xrightarrow {\op {id}} X]) = T_y(TX) \cap [X].$$
Hence such a natural transformation  is unique.
\end{thm}

\begin{rem} Before giving a quick proof of Theorem \ref{Ty*}, to avoid some possible question on the image of ${T_y}_*$ in Theorem \ref{Ty*} , it would be better to make a remark here. Even though the target of 
$$\displaystyle \widetilde {td_{*(y)}^{BFM}}: G_0(X) \otimes \bZ[y] \to H_*(X) \otimes \bQ[y, (1+y)^{-1}]$$ is 
$H_*^{BM}(X) \otimes \underline {\bQ[y, (1+y)^{-1}]}$, 
the image of ${T_y}_* = \widetilde {td_{*(y)}^{BFM}} \circ \Lambda _y^{mot}$ is in $H_*(X) \otimes \underline {\bQ[y]}$.  As mentioned before, it is because by Hironaka's resolution of singularities, induction on dimension, the normalization condition,  and the naturality of ${T_y}_*$, $K_0(\Cal {VAR}/X)$  is generated by $[V \xrightarrow {h} X]$ with $h$ being proper and $V$ being smooth. Hence
$${T_y}_*([V \xrightarrow {h} X]) = {T_y}_*(h_*[V \xrightarrow {\op {id}_V} V] ) = h_*({T_y}_*([V \xrightarrow {\op {id}_V} V] ) \in H_*^{BM}(X) \otimes \bQ[y] .$$
\end{rem}

\begin{proof} There is a slick way of proving this as in our paper \cite{BSY1} . Here we give a direct nonslick one. Let $X$ be smooth.
\begin {align*}
\hspace {2.5cm}& \widetilde {td_{*(y)}^{BFM}} \circ \Lambda _y^{mot}([X \xrightarrow {\op {id}} X]) \\
&= \widetilde {td_{*(y)}^{BFM}} (\lambda_y(\Omega_X)) \\
&= \sum_{i \geq 0}\frac{1}{(1+y)^i} td_{*i}^{BFM} (\lambda_y(\Omega_X)) \\
& = \sum_{i \geq 0} \frac{1}{(1+y)^i} \left (td_*^{BFM} (\lambda_y(\Omega_X))\right )_i \\
& = \sum_{i \geq 0} \frac{1}{(1+y)^i} \left (td_*^{BFM} (\lambda_y(T^*X) \otimes [\Cal O_X])\right )_i \\
& = \sum_{i \geq 0}\frac{1}{(1+y)^i} \left (ch(\lambda_y(T^*X)) \cap  td_*^{BFM}(\Cal O_X)\right )_i \\
& = \sum_{i \geq 0} \frac{1}{(1+y)^i} \left (ch(\lambda_y(T^*X)) \cap  (td(TX) \cap [X])\right )_i \\
& = \sum_{i \geq 0}\frac{1}{(1+y)^i} \left (\prod_{j=1}^{\op {dim}X} (1+ye^{-\alpha _j}) \prod_{j=1}^{\op {dim}X} \frac {\alpha_j}{1 - e^{-\alpha_j}} \right)_{\op {dim}X - i}  \cap [X].
\end{align*}
Furthermore we can see the following:
\begin {align*}
 \frac{1}{(1+y)^i} & \left (\prod_{j=1}^{\op {dim}X} (1+ye^{-\alpha _j}) \prod_{j=1}^{\op {dim}X} \frac {\alpha_j}{1 - e^{-\alpha_j}} \right)_{\op {dim}X - i}  \\
& = \frac{(1+y)^{\op {dim}X}} {(1+y)^i\quad \quad  } \left (\prod_{j=1}^{\op {dim}X} \frac{1+ye^{-\alpha _j}}{1 + y} \prod_{j=1}^{\op {dim}X} \frac {\alpha_j}{1 - e^{-\alpha_j}} \right)_{\op {dim}X - i} \\
& = (1+y)^{\op {dim}X - i} \left (\prod_{j=1}^{\op {dim}X} \frac{1+ye^{-\alpha _j}}{1 + y} \prod_{j=1}^{\op {dim}X} \frac {\alpha_j}{1 - e^{-\alpha_j}} \right)_{\op {dim}X - i} \\
& = \left (\prod_{j=1}^{\op {dim}X} \frac{1+ye^{-\alpha _j}}{1 + y} \prod_{j=1}^{\op {dim}X} \frac {\alpha_j (1+y)}{1 - e^{-\alpha_j(1+y)}} \right)_{\op {dim}X - i} \\
& = \left (\prod_{j=1}^{\op {dim}X} \frac{1+ye^{-\alpha _j}}{1 + y} \cdot \frac {\alpha_j (1+y)}{1 - e^{-\alpha_j(1+y)}} \right)_{\op {dim}X - i}  \\
& = \left (\prod_{j=1}^{\op {dim}X} \frac{\alpha_j(1 +y)}{1 - e^{-\alpha_j(1 +y)}} - \alpha_jy \right)_{\op {dim}X - i} \\
& = \left (T_y(TX) \right)_{\op {dim}X - i} 
\end{align*}
Therefore we get that $\widetilde {td_{*(y)}^{BFM}} \circ \Lambda _y^{mot}([X \xrightarrow {\op {id}} X])  =  T_y(TX) \cap [X].$
\end{proof}

Thus it remains to show Theorem \ref{lambda} and there are at least three proofs and each has its own advantage.\\

\noindent
{\bf[PROOF 1]:}
\underline {By Saito's Theory of Mixed Hodge Modules} \cite{Sai1, Sai2, Sai3, Sai4, Sai5, Sai6}: 

Even though Saito's theory is very complicated, this approach turns out to be useful and for example has been used in recent works of Cappell, Libgober, Maxim,  Sch\"urmann and Shaneson \cite{CLMS1, CLMS2, CMS1, CMS2, CMSS, MS, MS2}, related to intersection (co)homology. Here we recall only the ingredients which we need to define $\Lambda _y^{mot}$: 

\begin{enumerate}
\item [\underline {MHM1}]: To $X$ one can associate an abelian category
of {\em mixed Hodge modules\/} $MHM(X)$, together with a functorial
pullback $f^{*}$ and pushforward $f_{!}$ on the level of bounded derived categories
$D^{b}(MHM(X))$ for any (not necessarily proper) map.
These natural transformations are functors of triangulated categories.\\

\item [\underline {MHM2}]: Let $i:Y\to X$ be the inclusion of a closed subspace, 
with open complement $j: U:=X\backslash Y \to X$.
Then one has for $M\in D^{b}MHM(X)$ a distinguished triangle 
\[j_{!}j^{*}M \to M \to i_{!}i^{*}M \stackrel{[1]}{\to} \:.\]\

\item [\underline {MHM3}]: For all $p\in \bZ$ one has a ``filtered De Rham complex" functor of triangulated
categories
$$gr^{F}_{p}DR: D^{b}(MHM(X) )\to D^{b}_{coh}(X)$$
commuting with proper pushforward. Here $D^{b}_{coh}(X)$ is
the bounded derived category of sheaves of ${\Cal O}_{X}$-modules 
with coherent cohomology sheaves. Moreover, $gr^{F}_{p}DR(M)=0$
for almost all $p$ and $M\in D^{b}MHM(X)$ fixed.\\

\item [\underline {MHM4}]:
There is a distinguished element $\bQ_{pt}^{H}
\in MHM(pt)$ such that 
\[gr^{F}_{-p}DR(\bQ_{X}^{H}) \simeq \Omega^p_X[-p]
\in D^{b}_{coh}(X)\]
for $X$ smooth and pure dimensional.
Here $\bQ_{X}^{H}:=\pi_X^{*}\bQ_{pt}^{H}$ for $\pi_X: X\to pt$
a constant map, with $\bQ_{pt}^{H}$ viewed as a complex
concentrated in degree zero.\\
\end{enumerate}

The above transformations are functors of triangulated
categories, thus they induce functors even on the level of {\em Grothendieck groups 
of triangulated categories}, which we denote by the same name. Note that for these {\em Grothendieck groups\/} we have isomorphisms
\[K_{0}(D^{b}MHM(X)) \simeq K_{0}(MHM(X)) \quad \text{and}
\quad K_{0}(D^{b}_{coh}(X)) \simeq G_{0}(X)\]
by associating to a complex its alternating sum of cohomology objects.

Now we are ready to define the following two transformations $mH$ and $gr^{F}_{-*}DR$:

$$mH: K_{0}(\Cal {VAR}/X) \to K_{0}(MHM(X))\quad \text {defined  by} \quad 
mH([V \xrightarrow {f} X]) :=  [f_{!} \bQ_{V}^{H}].$$

In a sense $K_{0}(MHM(X))$ is like the abelian group of 
``mixed-Hodge-module constructible functions" with the class of $\bQ_{X}^{H}$
as a ``constant function'' on $X$. The well-definedness of $mH$, i.e., the additivity relation follows from the above properety (MHM2). By (MHM3) we get the following homomorphism commuting with proper pushforward

$$gr^{F}_{-*}DR: K_{0}(MHM(X)) \to G_{0}(X)\otimes \bZ[y,y^{-1}] $$
defined  by
$$gr^{F}_{-*}DR([M]) := \sum_{p} \: [gr^{F}_{-p}DR(M)]\cdot (-y)^{p} $$
Then we define our $\Lambda _y^{mot}$ by the composite of these two natural transformations:
$$\Lambda _y^{mot} := gr^F_{-*}DR \circ mH : K_0(\Cal {VAR}/X) \xrightarrow {mH} K_0(MHM(X)) \xrightarrow {gr^F_{-*}DR} G_0(X) \otimes \bZ[y].$$

By (MHM4), for $X$ smooth and pure dimensional we have that 
\[ gr^{F}_{-*}DR\circ mH([id_{X}]) =  
\sum_{p = 0}^{dim X} \; [\Omega_X^p]\cdot y^{p}
\in G_{0}(X)\otimes \bZ[y] \:.\]
\\
Thus we get the unique existence of the ``motivic" $\lambda_y$-class transformation
$\Lambda _y^{mot}$. \\

\noindent
 {\bf[PROOF 2]:} \underline {By the filtered Du Bois complexes} \cite{DB}: Recall that we have the following surjective homomophism
$$ A: Iso^{prop}(\Cal {VAR}/X) \twoheadrightarrow K_0(\Cal {VAR}/X) .$$
We can describe $ker A$ as follows:

\begin{thm} \label{lem:iso1}
$K_{0}(\Cal {VAR}/X)$ is isomorphic to the quotient of $Iso^{pro}(\Cal {VAR}/X)$  modulo the
``acyclicity'' relation
\begin{equation} \label{eq:ac}
[\emptyset \to X]=0 \quad \text{and} \quad
[\tilde{X}'\to X] - [\tilde{Z}'\to X]= [X'\to X] - [Z'\to X] \:,\tag{ac}
\end{equation}
for any cartesian diagram
\begin{displaymath} \begin{CD}
\tilde{Z}' @>>> \tilde{X}' \\
@VVV  @VV q V \\
Z' @> i >> X' @>>> X \:,
\end{CD} \end{displaymath}
with $q$ proper, $i$ a closed embedding,  and $q: \tilde{X}'\backslash \tilde{Z}'
\to X' \backslash Z'$ an isomorphism. 
\end{thm}

 For a proper map $X' \to X$, consider the filtered Du Bois complex
 $$ (\underline \Omega_{X'}^*, F),$$
 which satisfies that
 \begin{enumerate}
 \item
 $\underline\Omega_{X'}^*$ is a resolution of the constant sheaf $\bC$.
 
 \item
 $gr_F^p(\underline \Omega_{X'}^*) \in D_{coh}^b(X').$
 
 \item
Let $DR(\Cal O_{X'}) = \Omega_{X'}^*$ be the De Rham complex of $X'$ with $\sigma$ being the stupid filtration. Then there is a filtered morphism
$$ \lambda : (\Omega_{X'}^*, \sigma) \to (\underline \Omega_{X'}^*, F).$$
If $X'$ is smooth, then this is a filtered quasi-isomorphism.
\end{enumerate}

 Note that $G_0(X') \cong K_0(D_{coh}^b(X'))$.
 Let us define
 $$[gr_F^p(\underline \Omega_{X'}^*)] := \sum_i (-1)^i H^i(gr_F^p(\underline \Omega_{X'}^*) ) \in  K_0(D_{coh}^b(X')) = G_0(X').$$
 
 \begin{thm} The transformation
 $$\Lambda _y^{mot}: K_0(\Cal {VAR}/X) \to G_0(X) \otimes \bZ[y] \qquad \qquad $$
 defined by
 $$\Lambda _y^{mot}([X' \xrightarrow {h} X]) := \sum_p h_*[gr_F^p(\underline \Omega_{X'}^*)] (-y)^p $$
 is well-defined and is a unique natural transformation satisying the normalization condition that for smooth $X$
 $$\Lambda _y^{mot}([X \xrightarrow {\op {id}_X} X]) = \sum_{p =  0}^{dimX} [\Omega_X^p]y^p = \lambda_y(T^*X) \otimes \Cal O_X .$$
 \end{thm}
 
 \begin{proof} The well-definedness follows simply from the fact that $\Lambda _y^{mot}$ preserves the above ``acyclicity'' relation \cite{DB}. Then the uniqueness follows from resolution of singularities and the normalization conditon for smooth varieties. \\
 \end{proof}
 
\begin{rem}When $X$ is smooth, $[gr_{\sigma}^p(\underline \Omega_{X}^*)] = (-1)^p[\Omega_X^p]$ !
 That is why we need $(-y)^p$, instead of $y^p$, in the above definition of $\Lambda _y^{mot}([X' \xrightarrow {h} X]) $.
 \end{rem}
 
 \begin{rem} When $y=0$, we have the following natural transformation
 $$\Lambda_0^{mot}: K_0(\Cal {VAR}/X)  \to G_0(X) \quad \text {defined by } \quad \Lambda_0^{mot}([X' \xrightarrow {h} X]) = h_*[gr_F^0(\underline \Omega_{X'}^*)] $$
 satisying the normalization condition that for a smooth $X$
 $$\Lambda_0^{mot}([X \xrightarrow {\op {id}_X} X]) = [\Cal O_X].$$
 \end{rem}
 \vspace{0.5cm}
 
\noindent
{\bf [PROOF 3]:}
 \underline {By using Bittner's theorem on $K_0(\Cal {VAR}/X)$} \cite{Bi}: Herer we recall that we have the following surjective homomophism
$$ B: Iso^{prop}(\Cal {SM}/X) \twoheadrightarrow K_0(\Cal {VAR}/X) .$$
$ker B$ is identified by F. Bittner and E. Looijenga as follows \cite{Bi} :
\begin{thm}
The group $K_{0}(\Cal {VAR}/X)$ is isomorphic to the quotient of $Iso^{prop}(\Cal {SM}/X)$ (the free abelian group generated by the isomorphism classes of proper morphisms from smooth varieties to $X$) modulo the
``blow-up'' relation
\begin{equation} \label{eq:bl}
[\emptyset \to X]=0 \quad \text{and} \quad
[Bl_{Y}X'\to X] - [E\to X]= [X'\to X] - [Y\to X] \:,\tag{bl}
\end{equation}
for any cartesian diagram
\begin{displaymath} \begin{CD}
E @> i'>> Bl_{Y}X' \\
@VV q' V  @VV q V \\
Y @> i >> X' @> f >> X \:,
\end{CD} \end{displaymath}
with $i$ a closed embedding of smooth (pure dimensional) spaces and 
$f:X'\to X$ proper. Here $Bl_{Y}X'\to X'$ is the blow-up of $X'$ along 
$Y$ with exceptional divisor $E$. Note that all these spaces over $X$ are
also smooth (and pure dimensional and/or quasi-projective, if this is the case for $X'$ and $Y$).
\end{thm}

The proof of this Bittner's theorem requires Abramovich et al's ``Weak Factorisation Theorem" \cite{AKMW} (see also \cite{W}). \\

\begin{cor} \label{3proof}
\item{(1)}
Let $B_{*}: \Cal{VAR}/k \to \Cal {AB}$ be a functor from the category $var/k$ of
(reduced) seperated schemes of finite type over $spec(k)$ to the category of abelian groups,
which is covariantly functorial
for proper morphism, with $B_{*}(\emptyset):=\{0\}$. Assume we can associate to any 
(quasi-projective) smooth space $X\in ob(\Cal{VAR}/k )$ 
(of pure dimension) a distinguished element $$\phi _{X}\in B_{*}(X)$$
 such that $h_{*}(\phi  _{X'})=\phi _{X}$ for any isomorphism $h: X'\to X$.
 Then there exists a unique natural transformation
 $$\Phi : Iso^{prop}(\Cal {SM}/-) \to B_{*}(-)$$
satisfying the ``normalization'' condition that for any smooth $X$ 
$$\Phi([X \xrightarrow {\op {id}_X} X])=\phi _{X}.$$

\item{(2)}
Let $B_{*}: \Cal{VAR}/k \to \Cal {AB}$ and $\phi _X$ be as above and furthermore we assume that 
$$q_{*}(\phi  _{Bl_{Y}X})-i_{*}q'_{*}(\phi  _{E}) = \phi  _{X}-i_{*}(\phi  _{Y}) \in B_{*}(X)$$
for any cartesian blow-up diagram as in the above Bittner's theorem with $f=id_{X}$. Then there exists a unique natural transformation
$$\Phi : K_0(\Cal {VAR}/-) \to B_{*}(-)$$
satisfying the ``normalization'' condition that for any smooth $X$ 
$$\Phi([X \xrightarrow {\op {id}_X} X])=\phi _{X}.$$
\end{cor}

PROOF 3  of our ${T_y}_*$ uses (2) of the above Corollary \ref{3proof}  by considering the coherent sheaf 
$$\Omega_X^p \in G_0(X)$$ 
of a smooth $X$ as the distinguished element $\phi_X$ of a smooth $X$. Because it follows from M. Gros's work \cite{Gr} or the recent Guill\'en--Navarro Aznar's work \cite{GNA} that it satisfies the ``blow-up relation"
  $$q_{*}(\Omega^p_{Bl_{Y}X})-i_{*}q'_{*}(\Omega^p_{E}) = \Omega^p_{X}-i_{*}(\Omega^p_{Y}) \in G_0(X),$$
  which implies the following  ``blow-up relation" for the $\lambda_y$-class
$$q_{*}(\lambda_y (\Omega_{Bl_{Y}X}))-i_{*}q'_{*}(\lambda_y (\Omega_{E}) )= \lambda_y(\Omega_{X})-i_{*}(\lambda_y (\Omega_{Y})) \in G_0(X)\otimes \bZ[y].$$ 
Therefore (2) of the above Corollary \ref {3proof}  implies the following theorem. 
\begin{thm} The transformation
 $$\Lambda _y^{mot}: K_0(\Cal {VAR}/X) \to G_0(X) \otimes \bZ[y] \qquad \qquad $$
 defined by
 $$\Lambda _y^{mot}([X' \xrightarrow {h} X]) := h_*  \biggl (\sum_{p\geq 0} [\Omega_{X'}^p] y^p \biggr) ,$$
where $X'$ is smooth and $h:X' \to X$ is proper, is well-defined and is a unique natural transformation satisying the normalization condition that for smooth $X$
 $$\Lambda _y^{mot}([X \xrightarrow {\op {id}_X} X]) = \sum_{p =  0}^{dimX} [\Omega_X^p]y^p = \lambda_y(T^*X) \otimes \Cal O_X.$$
 \end{thm}

\begin{rem} \label{beta-remark}
The forementioned \, virtual \, Poincar\'e polynomial $\beta_t$ \, for the category \, $\Cal {VAR}(\bR)$ of real algeraic varieties is the unique homomorphism
$$\beta_t: K_0(\Cal {VAR}(\bR)) \to \bZ[t] \quad \text {such that} \quad \beta_t(\bR^1) = t$$
and $\beta_t(X) = P_t(X)$ the classical or usual topological Poincar\'e polynomial for compact nonsingular varieties. The proof of the existence of $\beta_i$, thus $\beta_t$,  also uses (2) of the above  Corollary \ref{3proof}  (see \cite{MP}). Speaking of the Poincar\'e polynomial $P_t(X)$, we emphasize that this polynoimal cannot be a ``topological counting"$c_{top}$ at all in the category of topological spaces, simply because the argument in the proof of Proposition \ref{c_top} does not work! Thus the Poincar\'e polynomial $P_t(X)$ is certainly a \emph{ multiplicative topological invariant}, but not \emph {an additive topological invariant}!
\end{rem}

\begin{rem}\label{beta}
The virtual Poincar\'e polynomial $\beta_t: K_0(\Cal {VAR}(\bR)) \to \bZ[t] $ is the \emph{unique} extension of the Poincar\'e polynomial $P_t(X)$ to arbitrary varieties. Here it should be noted that if we consider complex algebraic varieties, the virtual Poincar\'e polynomial 
$$\beta_t: K_0(\Cal {VAR}) \to \bZ[t] $$
is equal to the following motivic characteristic, using only the weight filtration,
$$w\chi(X) = \sum (-1)^i \op {dim}_{\bC} \left ( Gr_q^W H_c^i(X, \bC) \right) t^q,$$
because on any smooth compact complex algebraic variety $X$ they are all the same:
$$\beta_t(X) = P_t(X) = w\chi(X).$$
The last equalities follow from the purity of the Hodge structures on  $H^k(X,\bQ)$, i.e., the Hodge structures are of pure weight $k$.

This ``weight filtration" motivic characteristic $w\chi(X)$ is equal to the specialization $\chi_{-t, -t}$ of the Hodge--Deligne polynomial for $(u, v) = (-t, -t)$. This observation implies that there is \emph {\underline {no class version} of the (complex) virtual Poincar\'e \underline{polynomial}}
$\beta_t: K_0(\Cal {VAR}) \to \bZ[t] $, i.e., there is \underline{no} natural transformation
$$\natural : K_0(\Cal {VAR}/-) \to H_*^{BM}(-) \otimes \bZ[t] $$
such that 
\begin{itemize}
\item
for a smooth compact $X$
$$\natural ([X \xrightarrow {\op {id}_X} X]) = c\ell(TX) \cap [X]$$
for some multiplicative characteristic class of complex vector bundles,

\item 
$$\natural(pt) = \beta_t : K_0(\Cal {VAR}) \to \bZ[t] .$$
\end{itemize}
It is because that $\beta_t(X) = \chi_{-t, -t}(X)$ for a smooth compact complex algebraic variety $X$ (hence for all $X$), thus as in  Remark \ref{elliptic} one can conclude that $(-t, -t) = (-1, -1)$, thus $t$ has to be equal to $1$, i.e., $t$ cannot be allowed to vary. In other words, the only chance for such a class version is when $t = 1$, i.e., the Euler--Poincar\'e characteristic
$\chi: K_0(\Cal {VAR}) \to \bZ$. In which case, we do have the Chern class transformation
$$c_*:K_0(\Cal {VAR}/-) \to H_*^{BM}(-;\bZ).$$
This follows again from (2) of Corollary \ref{3proof} and the blow-up formula of Chern class \cite{Ful}.
\end{rem}
\begin{rem} The same discussion as in the above Remark \ref {beta} can be applied to the context of real algebraic varieties, i.e., the same example for real elliptic curves lead us to the conclusion that $t =1$ for $\beta_t$ satisfying the corresponding normalization condition for a normalized multiplicative characteristic class. This class has to be a polynomial in the Stiefel--Whitney classes, and we end up with the Stiefel--Whitney homology class $w_*$, which also satisfies the corresponding blow-up formula.
\end{rem}

\begin{rem} (``Poorest man's" motivic characteristic class) If we use the above much simpler covariant functor $Iso^{prop}(\Cal {SM}/X)$ (the abelian group of ``poorest man's motivic functions"), we can get the following ``poorest man's motivic characteristic class" for any characteristic class $c\ell$ of vector bundles as follows:\\Let $c\ell$ be \underline {any} characteristic class of vector bundles with the coefficient ring $K$. Then there exists a unique natural trnasformation
$$c\ell_*: Iso^{prop}(\Cal {SM}/-) \to H_*^{BM}(-) \otimes K$$
satisfying the normalization condition that for any smooth variety $X$
$$ c\ell_*([X \xrightarrow{id_{X}}  X]) = c\ell(TX) \cap [X].$$
There is a bivariant theoretical version of $Iso^{prop}(\Cal {SM}/X)$ (see \cite{Yokura-OBT}). For a general reference for the bivariant theory, see Fulton--MacPherson's AMS Memoirs \cite{FM}.\\
\end{rem}

\section {A ``unification" of Chern class, Todd class and L-class of singular varieties}

Our next task is to show that in a sense our motivic characteristic class ${T_y}_*$ gives rise to  a ``unification" of  MacPherson's Chern class, Baum--Fulton--MacPherson's Todd class (recalled in the previous section) and Cappell--Shaneson's L-class of singular varieties, which is another main theorem of \cite{BSY1}. \\

Here we recal these three:\\

\noindent
(1) \underline {MacPherson's Chern class } \cite{M1}:
\begin{thm}
There exists a unique natural transformation
$$c_*^{Mac}:F(-) \to H_*^{BM}(-)$$
such that for smooth $X$
$$c_*^{Mac}(\jeden_X) = c(TX) \cap [X].$$
Here $F(X)$ is the abelian group of constructible functions, which is a covariant functor with the pushforward $f_*:F(X) \to F(Y)$ for a proper morphism $f:X \to Y$ defined by
$$f_*(\jeden_W) (y) =  \chi_c(f^{-1}(y) \cap W).$$ 
$c_*^{Mac}(X) := c_*^{Mac}(\jeden_X)$ is called  MacPherson's Chern class of $X$ (or the Chern -- \, Schwartz--MacPherson class).
$$\chi(X) = \int_X c_*^{Mac}(X) .$$
\end{thm}

\noindent
(2) Once again, \underline {Baum--Fulton--MacPherson's Todd class, or Riemann--Roch} \cite{BFM}:
\begin{thm}
There exists a unique natural transformation
$$td_*^{BFM}:G_0(-) \to H_*^{BM}(-)\otimes \bQ$$
such that for  smooth $X$
$$td_*^{BFM}(\Cal O_X) = td(TX) \cap [X].$$
Here $G_0(X)$ is the Grothendieck group of coherent sheaves on $X$, which is a covariant functor with the pushforward $f_*:G_0(X) \to G_0(Y)$ for a proper morphism $f:X \to Y$ defined by
$$f_!({\Cal F}) =  \sum_j (-1)^j [R^j f_*\Cal F].$$ 
$td_*^{BFM}(X) := td_*^{BFM}(\Cal O_X)$ is called the Baum--Fulton--MacPherson Todd class of $X$.
$$p_a(X) = \chi(X, \Cal O_X) = \int_X td_*^{BFM}(X).$$
\end{thm}

\noindent
(3) \underline { Cappell--Shaneson's $L$-class} \cite{CS1, Sh} (cf. \cite{Y3})
\begin{thm}
There exists a unique natural transformation
$$L^{CS}_*:\Omega(-) \to H_*^{BM}(-)\otimes \bQ$$
such that for  smooth $X$
$$L^{CS}_*(\Cal {IC}_X) = L(TX) \cap [X].$$
Here $\Omega(X)$ is the abelian group of Youssin's cobordism classes of self-dual constructible complexes of sheaves on $X$.
$L^{GM}_*(X) := L^{CS}_*(\Cal {IC}_X)$ is the Goresky--MacPherson homology $L$-class of $X$.
$$\sigma^{GM}(X) = \int_X L^{GM}_*(X),$$
which is the Goresky--MacPherson theorem \cite{GM}.\\
\end{thm}

In the following sense  our motivic characteristic class transformation 
$${T_y} _*: K_0(\Cal {VAR}/X) \to H_*^{BM}(X)\otimes \bQ[y] $$ ``unifies" the above three well-known characteristic classes of singular varieties. This could be a kind of positive partial answer to MacPherson's question of \emph{whether there is a unified theory of characteristic classes of singular varieties}, which was posed in his survey talk \cite {M2} at the 9th Brazilian Colloquium on Mathematics in 1973. 

\begin{thm} (A ``unification" of Chern, Todd and homology $L$-classes of singular varieties) 
\item
\noindent  (y = -1):  There exists a unique natural transformation $\epsilon: K_0(\Cal {VAR}/-) \to F(-)$ such that for $X$ nonsingular $\epsilon([X \xrightarrow {\op {id}} X]) = \jeden_X.$ And the following diagram commutes
$$\xymatrix{K_0(\Cal {VAR}/X)  \ar[dr]_ {{T_{-1}}_*}\ar[rr]^ {\epsilon} && F(X) \ar[dl]^{c_*^{Mac}\otimes \bQ}\\& H_*^{BM}(X)\otimes \bQ \:.}$$\\
\item
\noindent  (y = 0) : There exists a unique natural transformation $\ga: K_0(\Cal {VAR}/-) \to  G_0(-)$ such that for $X$ nonsingular $\ga([X \xrightarrow {\op {id}} X]) = [\Cal O_X].$ And the following diagram commutes
$$\xymatrix{K_0(\Cal {VAR}/X)  \ar[dr]_ {{T_0}_*}\ar[rr]^ {\ga} &&  G_0(X) \ar[dl]^{td_*^{BFM}}\\& H_*^{BM}(X)\otimes \bQ \:.}$$\\
\item
\noindent  (y = 1) : There exists a unique natural transformation $sd: K_0(\Cal {VAR}/-) \to \Omega (-)$
such that for $X$ nonsingular $sd ([X \xrightarrow {\op {id}} X]) = \left [\bQ_X[\op {2dim}X] \right ].$ And the following diagram commutes
$$\xymatrix{K_0(\Cal {VAR}/X)  \ar[dr]_ {{T_1}_*}\ar[rr]^ {sd} && \Omega(X) \ar[dl]^{L^{CS}_*}\\& H_*^{BM}(X)\otimes \bQ \:.}$$\\
\end{thm}

When $y = -1, 0$, it is straightforward to show the above. But, when it comes to the case when $y =1$, it is not straightforward at all, in particular to show the existence of $sd: K_0(\Cal {VAR}/-) \to \Omega (-)$ is ``not obvious at all".  Another thing is that we have to go through some details of Youssin's work \cite{You} and in fact we need (2) of the above Corollary \ref{3proof} again. We do not know any other way of proving the existence of $sd: K_0(\Cal {VAR}/-) \to \Omega (-)$. For the details see \cite{BSY1} (see also \cite{BSY2, SY}).\\

Finally we make the following :
\begin{rem}
\begin{enumerate}
\item
(y = -1):  ${T_{-1}}_*(X) = c_*^{Mac}(X)\otimes \bQ$\\

\item
(y = 0):  In general,  for a singular variety $X$ we have
 $$\Lambda_0^{mot}([X \xrightarrow {\op {id}_X} X])  \not =  [\Cal O_X].$$
  Therefore, in general, ${T_0}_*(X) \not = td_*^{BFM}(X)$. So, our ${T_0}_*(X)$ shall be called the Hodge--Todd class and denoted by $td_*^{H}(X)$. However, if $X$ is a Du Bois variety, i.e., every point of $X$ is a Du Bois singularity (note a nonsingular point is also a Du Bois singularity), we DO have
  $$\Lambda_0^{mot}([X \xrightarrow {\op {id}_X} X])  =  [\Cal O_X].$$ 
  This is because of the definition of Du Bois variety: $X$ is called a Du Bois variety if we have
  $$ \Cal O_X = gr_{\sigma}^0(DR(\Cal O_X)) \cong gr_F^0(\underline \Omega_{X}^*).$$ 
 Hence, for a Du Bois variety $X$ we have ${T_0}_*(X) = td_*^{BFM}(X).$ For example, S. Kov\'acs \cite{Kov} proved Steenbrink's conjecture that rational singularities are Du Bois, thus for the quotient $X$ of any smooth variety acted on by a finite group we have that  ${T_0}_*(X)  = td_*^{BFM}(X)$.

\item
(y = 1): In general, $sd([X \xrightarrow {\op {id}_X} X]) \not = \Cal {IC}_X$, hence  ${T_1}_*(X) \not = L_*^{GM}(X)$. So, our ${T_1}_*(X)$ shall be called the Hodge--$L$-class and denoted by $L_*^{H}(X)$. A conjecture is that ${T_1}_*(X) \not = L_*^{GM}(X)$ for a rational homology manifold $X$. \\
\end{enumerate}
\end{rem}

\section {A few more conjectures}

\begin{con} \label{conjecture1}
Any natural transformation without the normalization condition 
$$T: K_0(\Cal {VAR}/X) \to H_*^{BM}(X)\otimes \bQ[y] $$
is a linear combination of components ${{td_y}_*}_i:K_0(\Cal {VAR}/X) \to H_{2i}^{BM}(X)\otimes \bQ[y]$:
$$T = \sum_{i \geq 0} r_i(y) \, {{td_y}_*}_i \quad (r_i (y)\in \bQ[y]).$$
\end{con}

This conjecture means that the normalization condition for smooth varieties imposed to get our motivic characteristic class can be \emph{basically dropped.}  This conjecture is motivated by the following theorems:

\begin{thm} (\cite{Y0})
Any natural transformation without the normalization condition  
$$T:G_0(-) \to H_*^{BM}(-)\otimes \bQ$$
 is a linear combination of components ${td_*^{BFM}}_i:G_0(-) \to H_{2i}^{BM}(-)\otimes \bQ$
$$T = \sum_{i \geq 0} r_i\, {td_*^{BFM}}_i \quad (r_i \in \bQ).$$
\end{thm}

\begin{thm}\label{kmy} (\cite{KMY})
Any natural transformation without the normalization condition 
$$T:F(-) \to H_*^{BM}(-)\otimes \bQ$$
is a linear combination of components ${c_*^{Mac}}_i \otimes \bQ:G_0(-) \to H_{2i}^{BM}(-)\otimes \bQ$ of the \underline {rationalized} MacPherson's Chern class $c_*^{Mac} \otimes \bQ$ (i.e., a linear combination of ${c_*^{Mac}}_i $ mod torsion):
$$T = \sum_{i \geq 0} r_i\, {c_*^{Mac}}_i \otimes \bQ \quad (r_i \in \bQ) .$$
\end{thm}

\begin{rem} The above Theorem \ref{kmy} certainly implies the uniqueness of such a transformation $c_*^{Mac} \otimes \bQ$ satisfying the normalization. The proof of Theorem \ref{kmy} DOES NOT appeal to the resolution of singularities at all, therefore modulo torsion the uniquess of the MacPherson's Chern class transformarion $c_*^{Mac}$ is proved without using the resolution of singularities. However, in the case of integer coefficients, as shown in \cite{M1}, the uniqueness of $c_*^{Mac}$ uses the resolution of singualrities and as far as the author knows, it seems that there is no proof available without using resolution of singularities. Does there exist any mysterious connection between resolution of singularities and finite torsion ? \footnote{A comment by J. Sch\"urmann: ``There is indeed a relation between resolution of singularities and torsion information: in \cite{Totaro} B. Totaro shows by resolution of singularities that the fundamental class $[X]$ of a complex algebraic variety $X$ lies in the image from the complex cobordism $\Omega^U(X) \to H_*(X, \bZ)$. And this implies some non-trivial topological restrictions, i.e., all odd-dimensional elemenets of the Steenrod algebra vanish on $[X]$ viewed in $H_*(X, \bZ_p)$."}
\end{rem}
Furthermore hinted by these two theorems, it would be natural to speculate the following ``linearity" on the Cappell--Shaneson's L-class also:\\

\begin{con}
Any natural transformation without the normalization condition 
$$T:\Omega(-) \to H_*^{BM}(-)\otimes \bQ$$
is a linear combination of components ${_*^{CS}}_i :\Omega(-) \to H_{2i}^{BM}(-)\otimes \bQ$:
$$T = \sum_{i \geq 0} r_i\, {L_*^{CS}}_i \quad (r_i \in \bQ) .$$\\
\end{con} 

\section{Some more remarks}

For complex algebraic varieties there is another important homology theory. That is Goresky--MacPherson's {\bf Intersection Homology Theory} $IH$ \cite{GM}(see also \cite{KW}), which satifsies all the properties which the ordinary (co)homology theory for nonsingular varieties have, in particular the Poincar\'e duality holds, in contrast to the fact that in general it fails for the ordinary (co)homology theory  of singular varieties. In order that the Poincar\'e duality theorem holds, one needs to control cycles according to \emph{perversity}, which is sensitive to, or ``control",  complexity of singularities. M. Saito showed that $IH$ satisfies pure Hodge structure just like the cohomology satisfies the pure Hodge structure for compact smooth manifolds (see also \cite{CaMi1,CaMi2}). In this sense, $IH$ is a convenient gadget for possibly singular varieties, and using the $IH$, we can also get various invariants which are sensitive to the structure of given possibly singular varieties. For the history of $IH$, see Kleiman's survey article \cite{Kl}, and for {\bf $L_2$-cohomology }very closely related to the intersection homology, e.g., see \cite{CGM, Go, Lo1, SS, SZ}. 
Thus for the category of compact complex algebriac varieties two competing machines are available:\
\begin{itemize}

\item {\bf Ordinary (Co)homology \, + \, Mixed Hodge Structures}. \
 \item{\bf Intersection Homology \, +  \, Pure Hodge Structures}.\
\end{itemize}

Of course, they are the same for the subcategory of compact smooth varieties.\\
 
So, for singular varieties one can introduce the similar invariants using $IH$; in other words, one can naturally think of the $IH$-version of the Hirzebruch $\chi_y$ genus, because of the pure Hodge structure, denote by $\chi^{IH}_y$: Thus we have invariants $\chi_y$-genus and $\chi^{IH}_y$-genus. As to the class version of these, one should go through the derived category of mixed Hodge modules, because the intersection homology sheaf lives in it. Then it is obvious that the difference between these two genera or between the class versions of theses two genera should come from the singularities of the given variety. For such an investigation, see Cappell--Libgober--Maxim--Shaneson \cite{CMS1, CMS2, CLMS1, CLMS2}. \\

The most important result is the so-called ``Decomposition Theorem" of Beilinson--Bernstein--Deligne--Gabber \cite{BBDG}, \, which was conjectured by I.  M. \, Gelfand and R. MacPherson. A more geometric proof of this is given in the above mentioned paper \cite{CaMi1} of M. de Cataldo and L. Migliorini.\\
 
Speaking of the intersection homology, the general category for $IH$ is the catgeory of pseudo-manifolds and the canonical and well-studied invariant for pseudo-manifolds is the signature, because of the Poincar\'e duality of $IH$. Banagl's monograph \cite{Ba} is recommended on this topic and also see \cite{Ba2, Ba3, Ba4, BCS, CSW, CW, Wei} etc.. Very roughly speaking, ${T_y}_*$ is a kind of ``deformation" or ``perturbation" of Baum--Fulton--MacPherson's Riemann--Roch.  It would be interesting to consider a similar kind of ``deformation" of $L$-class theory defined on the (co)bordism theory of pseudo-manifolds. \footnote{A commnet by J. Sch\"urmann: ``A deformation of the $L$-class theory seems not reasonable. Only the signature = $\chi_1$-genus factorizes over the oriented cobordism ring $\Omega^{SO}$, so that this invariant is of more topological nature related to stratified spaces. For the other looking-for (``deformation") invariants one needs a complex algebraic or analytic structure. So what is missing up to now is a suitable theory of almost complex stratified spaces."} \\
 
Finally, since we started the present paper with counting, we end with posing the following question:
how about counting psuedo-manifolds respecting the structure of psuedo-manifolds, i.e., 
$$\text {\bf Does ``stratified counting" $c_{stra}$ make a sense  ?}$$ 
For complex algebraic varieties, which are psuedo-manifolds, the algebraic counting $c_{alg}$ (using Mixed Hodge Theory = Ordinary (Co)homology Theory + Mixed Hodge Structure) in fact ignores the stratification. So, in this possible problem, one should consider Intersection Homology + Pure Hodge Structure, although the intersection homology \emph {is a topological invariant, thus in particular independent of the stratification}. \\

J. Sch\"urmann provides one possible answer for the above question -- \emph{Does ``stratified counting" $c_{stra}$ make a sense  ?} (Since it is a bit long, it is cited just below, not as a footnote.):
\begin{itemize}
\item[] ``One possible answer would be to work in the complex algebraic context with a fixed (Whitney) stratification $X_{\bullet}$, so that the closure of a stratum $S$ is a union of strata again. Then one can work with the Grothendieck group $K_0(X_{\bullet})$ of $X_{\bullet}$-constructible sets, i.e., those which are a union of such strata. The topological additive counting would be related again to the Euler characteristic and the group $F(X_{\bullet})$ of $X_{\bullet}$-constructible functions. A more sophisticated version is the Grothendieck group $K_0(X_{\bullet})$ of $X_{\bullet}$-constructible sheaves (or sheaf complexes). These are generated by classes $j_!L_S$ for $j: S \to X$ , the inclusion of a stratum $S$, and $L_S$ a local system on $S$, and also by the intermediate extensions ${j_!}_*L_S$, which are perverse sheaves. In relation to signature and duality, one can work with the corresponding cobordism group $\Omega(X_{\bullet})$ of Verdier self-dual $X_{\bullet}$-constructible sheaf complexes. These are generated by ${j_!}_*L_S$, with $L_S$ a self-dual local system on $S$. Finally one can also work with the Grothendieck group $K_0(MHM(X_{\bullet}))$ of mixed Hodge modules, whose underlying rational complex is $X_{\bullet}$-constructible. This last group is of course not a topolological invariant."\\
\end{itemize}

We hope to come back to the problem of a possible ``stratified counting" $c_{stra}$.\\

\noindent
{\bf Acknowledgements:} This paper is based on the author's talk at the workshop ``Topology of Stratified Spaces" held at MSRI, Berkeley, September 8 - 12. He woud like to thank the organizers (Greg Friedman, Eug\'enie Hunsicker, Anatoly Libgober, and Laurentiu Maxim) for  inviting him to the workshop. He would also like to thank the referee and J\"org Sch\"urmann for their careful reading the paper and valuable comments and suggestions, and G. Friedmann and L. Maxim for their valuable comments and suggestions on an earlier version of the paper.

\end{document}